\documentclass[11pt]{article}
\usepackage{fullpage}
\usepackage{fancyhdr}
\usepackage{amsmath,amssymb,amsthm}
\usepackage{amsmath,amscd}
\usepackage{mathtools}
\usepackage{mathrsfs}
\usepackage{tikz-cd}
\usepackage{authblk}
\usepackage{enumitem}
\usepackage{ stmaryrd }

\usepackage[letterpaper, headsep=1.5cm, margin=1.2in]{geometry}

\usepackage{hyperref}
\usepackage[symbols,nogroupskip,sort=use]{glossaries-extra}

\newcommand{\Z}{\mathbb{Z}}

\newcommand{\Q}{\mathbb{Q}}

\newcommand{\gothm}{\mathfrak m}

\newcommand{\boldm}{\mathbf{m}}

\newcommand{\Sym}{\text{Sym}}

\newcommand{\Spec}{\text{Spec}}
\newcommand{\Proj}{\text{Proj}}

\newcommand{\Div}{\text{Div}}

\newcommand{\ang}[1]{\langle #1 \rangle}

\renewcommand {\bar}{\overline}

\newcommand{\diag}[1]{\mathbf{diag}(#1)}
\newcommand{\cyc}[1]{\mathbf{cyc}(#1)}
\newcommand{\tcyc}[1]{\mathbf{tcyc}(#1)}

\AtBeginDocument{}

\newtheorem{theorem}{Theorem}[section]
\newtheorem{lemma}[theorem]{Lemma}

\newtheorem{proposition}[theorem]{Proposition}

\newtheorem{corollary}[theorem]{Corollary}

\theoremstyle{definition}

\newtheorem{example}[theorem]{Example}
\newtheorem{remark}[theorem]{Remark}
\newtheorem{definition}[theorem]{Definition}

\def\frak{\relaxnext@\ifmmode\let\next\frak@\else
	\def\next{\Err@{Use \string\frak\space only in math mode}}\fi\next}
\def\goth{\relaxnext@\ifmmode\let\next\frak@\else
	\def\next{\Err@{Use \string\goth\space only in math mode}}\fi\next}
\def\frak@#1{{\frak@@{#1}}}
\def\frak@@#1{\noaccents@\fam\euffam#1}
\font\tengoth=eufm10
\newfam\gothfam \def\goth{\fam\gothfam\tengoth} \textfont\gothfam=\tengoth

\makenoidxglossaries

\glsxtrnewsymbol[description={A smooth proper curve over $\mathbb{F}_q$.}]{X}{\ensuremath{X}}
\glsxtrnewsymbol[description={An Artin character on $X$.}]{rho}{\ensuremath{\rho}}

\glsxtrnewsymbol[description={The number of points in $X$ where $\rho$ is ramified.}]{gothm}{\ensuremath{\mathbf{m}}}

\glsxtrnewsymbol[description={The points of $X$ where $\rho$ is ramified.}]{taus}{\ensuremath{\tau_1,\dots,\tau_{\mathbf{m}}}}

\glsxtrnewsymbol[description={The Swan conductor of $\rho$ at $Q$.}]{swan}{\ensuremath{s_Q}}

\glsxtrnewsymbol[description={The exponent of $\rho$ at $Q$.}]{exponent}{\ensuremath{\mathbf{e}_Q}}

\glsxtrnewsymbol[description={The natural number between $0$ and $q-2$ such that $\frac{\epsilon_Q}{q-1}$ represents $\mathbf{e}_Q$.}]{special exponent}{\ensuremath{\epsilon_Q}}

\glsxtrnewsymbol[description={The sum of the $p$-adic digits of $\epsilon_Q$.}]{exponent digits}{\ensuremath{\omega_Q}}

\glsxtrnewsymbol[description={The sum $\sum\limits_{Q\in W}\frac{1}{a(p-1)}\omega_Q$.}]{tame sum}{\ensuremath{\Omega_\rho}}

\glsxtrnewsymbol[description={The Hodge polygon associated to $\rho$.}]{Hodge polygon}{\ensuremath{HP(\rho)}}

\glsxtrnewsymbol[description={The Amice ring with parameter $t$.}]{Amice}{\ensuremath{\mathcal{E}_t}}

\glsxtrnewsymbol[description={The bounded Robba ring with parameter $t$.}]{Bounded Robba}{\ensuremath{\mathcal{E}_t^\dagger}}

\glsxtrnewsymbol[description={The ring of bounded functions that converge in the annulus $0<v_p(t)\leq r$.}]{r-convergent functions}{\ensuremath{\mathcal{E}(0,r]}}

\glsxtrnewsymbol[description={The affine curve $\eta^{-1}(\mathbb{P}^1-\{0,1,\infty\})$ contained in $X$.}]{V}{\ensuremath{V}}

\glsxtrnewsymbol[description={A formal lift of $V$.}]{V formal}{\ensuremath{\mathbf{V}}}

\glsxtrnewsymbol[description={The rigid fiber of $\mathbf{V}$.}]{V rigid}{\ensuremath{\mathcal{V}^{rig}}}

\glsxtrnewsymbol[description={The ring of `overconvergent' functions on $\mathbf{V}$.}]{integral overconvergent on V}{\ensuremath{B^\dagger}}

\glsxtrnewsymbol[description={The ring of `overconvergent' functions on $\mathcal{V}^{rig}$.}]{overconvergent on V}{\ensuremath{\mathcal{B}^\dagger}}

\glsxtrnewsymbol[description={$\mathbf{r}=(r_Q)_{Q \in W}$ is a tuple of positive rational numbers. Each $r_Q$ represents
a radius of convergence around the point $Q$.}]{radius tuple}{\ensuremath{\mathbf{r}}}

\glsxtrnewsymbol[description={The ring of functions on $\mathcal{V}$
		converge on the annulus $0<v_p(u_Q)\leq r_Q$ for each $Q \in W$.}]{r converge on V}{\ensuremath{\mathcal{B}(0,\mathbf{r}]}}

\glsxtrnewsymbol[description={A morphism from $X$ to $\mathbb{P}^1$ with specific ramification properties.}]{morphism to P1}{\ensuremath{\eta}}

\glsxtrnewsymbol[description={The points of $X$ lying above $\{0,1,\infty\}$, i.e. $\eta^{-1}(\{0,1,\infty\})$.}]{W}{\ensuremath{W}}

\glsxtrnewsymbol[description={The ramification index of $\eta$ at $Q\in X$.}]{eQ}{\ensuremath{e_Q}}

\glsxtrnewsymbol[description={A subspace of $\bigoplus\limits_{j=0}^{a-1} \mathcal{O}_{\mathcal{E}^\dagger}$ depending on a tame exponent $\mathbf{e}$ and a Swan conductor $s$.}]{special local for type 1}{\ensuremath{\mathcal{D}_{\mathbf{e},s}}}

\glsxtrnewsymbol[description={A subspace of $\ \mathcal{O}_{\mathcal{E}^\dagger}$ used for studying $U_p$ for the second type of Frobenius endomorphism.}]{special local for type 2}{\ensuremath{\mathcal{D}}}

\glsxtrnewsymbol[description={The $p$-Frobenius structure associated to $\rho^{wild}$.}]{wild p-frob structure}{\ensuremath{\alpha_0}}

\glsxtrnewsymbol[description={The $q$-Frobenius structure associated to $\rho^{wild}$.}]{wild frob structure}{\ensuremath{\alpha}}

\glsxtrnewsymbol[description={The $p$-Frobenius structure associated to $\bigoplus\limits_{j=0}^{a-1}\chi^{\otimes p^j}$.}]{tame p-frob structure}{\ensuremath{N_0}}

\glsxtrnewsymbol[description={The $q$-Frobenius structure associated to $\bigoplus\limits_{j=0}^{a-1}\chi^{\otimes p^j}$.}]{tame frob structure}{\ensuremath{N}}

\glsxtrnewsymbol[description={The well behaved $p$-Frobenius structure associated to $\rho_Q$.}]{local nice frobenius structure}{\ensuremath{C_Q}}

\glsxtrnewsymbol[description={The `change of basis' element between the wild global and local $p$-Frobenius structures at $Q$.}]{wild local change of basis}{\ensuremath{b_Q}}

\glsxtrnewsymbol[description={The `change of basis' matrix between the tame global and local $p$-Frobenius structures at $Q$.}]{tame local change of basis}{\ensuremath{M_Q}}

\glsxtrnewsymbol[description={The subspace of functions at $Q$ with precise growth conditions. The growth conditions are determined by $\eta(Q)$ and the ramification datum at $Q$.}]{special convergence space}{\ensuremath{\mathcal{O}_{\mathcal{R}_Q}^{con}}}

\glsxtrnewsymbol[description={$a$ copies of the bounded Robba ring with parameter $u_Q$: $\bigoplus\limits_{j=0}^{a-1} \mathcal{E}_Q^\dagger$.}]{bounded robba in uQ}{\ensuremath{\mathcal{R}_Q^\dagger}}

\glsxtrnewsymbol[description={$a$ copies of the integral Robba ring with parameter $u_Q$: $\bigoplus\limits_{j=0}^{a-1} \mathcal{O}_{\mathcal{E}_Q^\dagger}$.}]{integral bounded robba in uQ}{\ensuremath{\mathcal{O}_{\mathcal{R}_Q^\dagger}}}

\glsxtrnewsymbol[description={The endomorphism of $\mathcal{R}$ acting on the $Q$ coordinate by $C_Q$.}]{C}{\ensuremath{C}}

\glsxtrnewsymbol[description={The endomorphism of $\mathcal{R}$ acting on the $Q$ coordinate by $\diag{b_Q}$.}]{b}{\ensuremath{b}}

\glsxtrnewsymbol[description={The endomorphism of $\mathcal{R}$ acting on the $Q$ coordinate by $M_Q$.}]{M}{\ensuremath{M}}

\glsxtrnewsymbol[description={The direct sum over $Q\in W$ of each $\mathcal{R}_Q$.}]{R}{\ensuremath{\mathcal{R}}}

\glsxtrnewsymbol[description={The direct sum over $Q\in W$ of each $\mathcal{R}_Q^\dagger$.}]{Rdagger}{\ensuremath{\mathcal{R}^\dagger}}

\glsxtrnewsymbol[description={Truncated elements of $\mathcal{R}_Q$ depending on $\eta(Q)$.}]{RtrunQ}{\ensuremath{\mathcal{R}^{trun}_Q}}

\glsxtrnewsymbol[description={The direct sum over $Q\in W$ of each $\mathcal{R}_Q^{trun}$.}]{Rtrun}{\ensuremath{\mathcal{R}^{trun}}}

\glsxtrnewsymbol[description={The overconvergent elements of $\mathcal{R}^{trun}$.}]{Rtrundagger}{\ensuremath{\mathcal{R}^{\dagger,trun}}}

\glsxtrnewsymbol[description={The direct sum of $\mathcal{O}_{\mathcal{R}_Q}^{con}$ over $Q \in W$.}]{OconR}{\ensuremath{\mathcal{O}_{\mathcal{R}}^{con}}}

\title{$p$-adic estimates of abelian Artin $L$-functions on curves}
\author{Joe Kramer-Miller}

\date{}

\begin{document}

	\title{$p$-adic estimates of abelian Artin $L$-functions on curves}
	\author{Joe Kramer-Miller}

	\date{}

		\maketitle
		
			\begin{abstract}
			The purpose of this article is to prove a ``Newton over Hodge'' result for 
			finite characters on curves. 
			Let $X$ be a smooth proper curve over a finite 
			field $\mathbb{F}_q$ of characteristic $p\geq 3$ and let $V \subset X$ be an 
			affine curve. Consider a nontrivial finite character $\rho:\pi_1^{et}(V) \to \mathbb{C}^\times$. In this article, we prove a lower bound on the Newton 
			polygon of the $L$-function $L(\rho,s)$. The estimate depends on monodromy invariants
			of $\rho$: the Swan conductor and the local exponents. Under
			certain nondegeneracy assumptions this lower bound agrees with the irregular Hodge
			filtration introduced by Deligne. In particular, our result further
			demonstrates Deligne's prediction that the irregular Hodge filtration would force
			$p$-adic bounds on $L$-functions.
			As a corollary, we obtain estimates on the 
			Newton polygon of a curve with a cyclic action in terms of monodromy invariants. 
		\end{abstract}

		\tableofcontents
		\section{Introduction}
		Let $p$ be a prime with $p\geq 3$ and let $q=p^a$. 
		Let $\gls{X}$ be a smooth proper curve of genus $g$ defined over $\mathbb{F}_q$
		with function field $K(X)$. We define $G_X$ to be
		the absolute Galois group of $K(X)$. Let $\gls{rho}:G_X \to \mathbb{C}^\times$ be
		a non-trivial continuous character.
		The $L$-function associated to $\rho$ is defined by
		\begin{align}
		L(\rho,s) &=\prod\frac{1}{1 - \rho(Frob_x) s^{\deg(x)}}, \label{introduction of L-function}
		\end{align}
		with the product taken over all closed points $x \in X$ where $\rho$ is unramified.
		By the Weil conjectures for curves (see \cite{Weil-his_conjectures}) we know that 
		\begin{align*}
		L(\rho,s) &= \prod_{i=1}^{d}(1-\alpha_i 
		s) \in \overline{\Z}[s].
		\end{align*}  
		It is then natural to ask what we can say about the algebraic integers $\alpha_i$.
		The Riemann hypothesis for curves tell us that $|\alpha_i|_\infty = \sqrt{q}$ for each
		Archimedean place. Furthermore, we know that the $\alpha_i$ are $\ell$-adic units for any prime $\ell\neq p$.
		This leaves us with the question: what are the $p$-adic valuations of the $\alpha_i$? 
		
		The purpose of this article is to study the $p$-adic properties of $L(\rho,s)$. 
		We prove a 
		``Newton over Hodge'' result. This is in the vein of a celebrated theorem of Mazur (see \cite{Mazur-NoverH}), which compares the Newton and Hodge polygons of an algebraic variety over $\mathbb{F}_q$. Our result differs from Mazur's, in that
		we study cohomology with coefficients in a local system. Our Hodge bound is defined using
		two monodromy invariants: the \emph{Swan conductor} and the \emph{tame exponents}. 
		The representation $\rho$ is analogous to a rank one differential equation on
		a Riemann surface with regular singularities twisted by an exponential differential equations (i.e. a
		weight zero twisted Hodge module in the language of Esnault-Sabbah-Yu in \cite{Esnault-Yu-irrHodge}). In this context
		one may define an irregular Hodge polygon (see \cite{Deligne-Malgrange-Ramis} or \cite{Esnault-Yu-irrHodge}). The
		irregular Hodge polygon agrees with the Hodge polygon we define under certain
		nondegeneracy hypotheses. Our
		result thus gives further credence to the philosophy that characteristic zero
		Hodge-type phenomena forces $p$-adic bounds on lisse sheaves in characteristic $p$.

		\subsection{Statement of main results} \label{subsection: statement of main results}
		To state our main result, we first introduce some monodromy invariants. The character $\rho$ factors uniquely as $\rho=\rho^{wild} \otimes \chi$,
		where $|Im(\rho^{wild})|=p^n$ and $|Im(\chi)|=N$ with $\gcd(N,p)=1$. 
		\begin{enumerate}
			\item (local) Let $Q \in X$ be a closed point. After
			increasing $q$ we may assume that $Q$ is an $\mathbb{F}_q$-point. 
			Let $u_Q$ be a local parameter at $Q$. Then $\rho$ restricts to
			a local representation $\rho_Q:G_Q \to \mathbb{C}^\times$, where
			$G_Q$ is the absolute Galois group of $\mathbb{F}_q((u_Q))$. 
			We let $\rho^{wild}_Q$ (resp. $\chi_Q$) denote
			the restriction of $\rho^{wild}$ (resp. $\chi$) to $G_Q$. 
			\begin{enumerate}
				\item (Swan conductors) Let $I_Q \subset G_Q$ be the inertia subgroup
				at $Q$. There is a decreasing filtration of subgroups $I_Q^s$ on $I_Q$, indexed by
				real numbers $s\geq 0$. The \emph{Swan conductor} at $Q$
				is the infimum of all $s$ such that $I_Q^s \subset \ker(\rho_Q)$ (see \cite[Chapter 1]{Katz-Kloosterman}). We denote
				the Swan conductor by $\gls{swan}$. Note that $s_Q=0$ if and only if $\rho_Q^{wild}$
				is unramified. 
				\item (Tame exponents) 
				After increasing $q$ we may assume $\chi_Q$ is totally ramified at $Q$. 
				There exists $e_Q \in \frac{1}{q-1}\Z$ such that $G_Q$
				acts on $t_Q^{e_Q}$ by $\chi_Q$. Note that $e_Q$ is unique up
				to addition by an integer. 
				\begin{itemize}
					\item The \emph{exponent} of $\chi$ at $Q$ is the equivalence class $\gls{exponent}$
					of $e_Q$ in $ \frac{1}{q-1}\Z/ \Z$.
					\item We define $\gls{special exponent}$ to be the unique integer between $0$ and
					$q-2$ such that $\frac{\epsilon_Q}{q-1} \in \mathbf{e}_Q$.
					\item Write $\epsilon_Q=e_{Q,0} + e_{Q,1}p + \dots +e_{Q,a-1} p^{a-1}$,
					where $0 \leq e_{Q,i}\leq p-1$. We define $\gls{exponent digits}=\sum e_{Q,i}$, the
					sum of the $p$-adic digits of $\epsilon_Q$. Note that $\omega_Q=0$
					if and only if $\chi_Q$ is unramified.
				\end{itemize}
			\end{enumerate}
		
		We refer to the tuple $R_Q=(s_Q, \mathbf{e}_Q, \epsilon_Q, \omega_Q)$ as
		a \emph{ramification datum} and $T_Q=(\mathbf{e}_Q,\epsilon_Q,\omega_Q)$
		as a \emph{tame ramification datum}. We define the sets
		\begin{align*}
			S_{Q} &= \begin{cases}
				\emptyset & s_Q=0 \\
				\Big\{ \frac{1}{s_{Q}}, \dots, 
				\frac{s_{Q}-1}{s_{Q}}\Big\} & s_Q \neq 0 \text{ and } \omega_Q =0 \\
				\Big\{ \frac{1}{s_{Q}}- \frac{\omega_{Q}}{as_{Q}(p-1)}, \dots, 
				\frac{s_{Q}}{s_{Q}}- \frac{\omega_{Q}}{as_{Q}(p-1)}\Big\} & s_Q \neq 0 \text{ and } \omega_Q \neq 0 
			\end{cases}.
		\end{align*}
			\item (global) Let $\gls{taus}$ be the points at which
			$\rho$ ramifies and let $\mathbf{n}\leq \gls{gothm}$ be such that $\tau_1, \dots, \tau_{\mathbf{n}}$
			are the points at which $\chi$ ramifies. We define 
			\begin{align*}
			\gls{tame sum} &= \frac{1}{a(p-1)} \sum_{i=1}^\mathbf{n} \omega_{\tau_i}.
			\end{align*} 
			This is a global invariant built up from the $p$-adic properties of the local exponents.
			One can show that $\Omega_\rho \in \Z_{\geq 0}$ (see \S \ref{subsubsection: global frobenius structure tame}). 
		\end{enumerate}
		Using these invariants we define the Hodge polygon $\gls{Hodge polygon}$
		to be the polygon whose slopes are  
		\begin{align*} 
		\big 
		\{\underbrace{0,\dots,0}_{g-1+\boldm-\Omega_\rho}
		\big \}\sqcup \big 
		\{\underbrace{1,\dots,1}_{g-1+\boldm-\mathbf{n} + \Omega_\rho}
		\big \}
		\sqcup \Big ( \bigsqcup_{i=1}^\boldm S_{\tau_i} \Big ),
		\end{align*}
		where $\sqcup$ denotes a disjoint union. We can now state our main result.
		\begin{theorem} \label{main theorem}
			The $q$-adic Newton polygon $NP_q(L(\rho,s))$ lies above the Hodge
			polygon $HP(\rho)$.
		\end{theorem}
		\begin{remark}
			It is worth mentioning that $HP(\rho)$ and $NP_q(L(\rho,s))$ have the
			same endpoints. To see this, first note that the $x$-coordinates of the endpoints of
			both polygons are $g-1+\mathbf{m} + \sum s_Q$. For $NP_q(L(\rho,s))$ this
			follows from the Euler-Poincare formula (see \cite[\S 2.3.1]{Katz-Kloosterman})
			and for $HP(\rho)$ this is clear from the definition. 
			Next, let $(s_{\tau_i}', \mathbf{e}_{\tau_i}', \epsilon_{\tau_i}', \omega_{\tau_i}')$ be
			the ramification datum associated to $\rho^{-1}$ at ${\tau_i}$. 
			Then 
			we have $s_{\tau_i}'=s_{\tau_i}$ and $\mathbf{e}_{\tau_i}'=-\mathbf{e}_{\tau_i}$. From this we see that
			$\omega_{\tau_{i}}'=a(p-1)-\omega_{\tau_{i}}$ for $1\leq i \leq \mathbf{n}$
			and $\omega_{\tau_i}'=0$ for $i>\mathbf{n}$, which
			implies $\Omega_{\rho^{-1}}=\mathbf{n}-\Omega_{\rho}$. Thus,
			for every slope $\alpha$ of $HP(\rho)$, there is a corresponding
			slope $1-\alpha$ of $HP(\rho^{-1})$. Similarly, by Poincare duality, we know that for every slope $\alpha$ of $NP_q(L(\rho,s))$, there is a corresponding slope $1-\alpha$ of $NP_q(L(\rho^{-1},s))$. It follows that the $y$-coordinates of the endpoints of $HP(\rho)\sqcup HP(\rho^{-1})$ and $NP_q(L(\rho,s))\sqcup NP_q(L(\rho^{-1},s))$
			agree. By applying Theorem \ref{main theorem} to $\rho$ and $\rho^{-1}$, we see that the endpoints
			of $HP(\rho)$ and $NP_q(L(\rho,s))$ are the same. 
		\end{remark}
	
		\begin{remark} \label{remark: previous work 1}
			When $\rho$ factors through an Artin-Schreier cover, Theorem \ref{main theorem}
			is due to previous work of the author (see \cite{kramermiller-padic}). 
		\end{remark}
		\begin{remark} \label{remark: previous work 2}
			The only other case where parts of Theorem \ref{main theorem} were
			previously known is when $X=\mathbb{P}^1$ and $\rho$ is unramified
			outside of $\mathbb{G}_m$. 
			Work of Adolphson-Sperber
			(see \cite{Adolphson-Sperber-twisted1} and \cite{Adolphson-Sperber-twisted2})
			studies the case where $|Im(\rho)|=pN$ and $\gcd(p,N)=1$. We note that
			the work of Adolphson and Sperber treats the case of higher dimensional tori as well. 
			These groundbreaking methods were applied to the case when $\rho$ is totally wild by Liu and Wei in \cite{Liu-Wei-witt-coverings}, introducing ideas from Artin-Schreier-Witt theory.
			For $\rho$ with arbitrary image there are some results by Liu (see \cite{Liu-Heilbronn_sums}),
			under strict conditions on the wild part of $\rho$ (this case corresponds to Heilbronn sums). 
		\end{remark}
		To the best of our knowledge, Theorem \ref{main theorem} was
		completely unknown outside of the situations described in Remark \ref{remark: previous work 1} and Remark \ref{remark: previous work 2}.
		
		\begin{example}
			Let $X=\mathbb{P}^1_{\mathbb{F}_q}$ and let $\tau_1,\dots,\tau_{4}$ be
			the points where $\rho$ ramifies.  Assume $|Im(\rho)|=2p^n$ and that $\rho$ is totally ramified at each $\tau_i$ (i.e. the inertia group at $\tau_i$ is equal to $Im(\rho)$). Let $f:E \to X$ be the
			genus $1$ curve over which $\chi$ trivializes and let $\upsilon_i=f^{-1}(\tau_i)$. 
			Consider the restriction $\rho_E=\rho|_{G_E}$. 	Let $(s_i, \mathbf{e}_i, \epsilon_i, \omega_i)$ be the ramification datum of $\rho$ at $\tau_i$ and let $(s_i', \mathbf{e}_i', \epsilon_i', \omega_i')$ be the ramification datum of $\rho_E$ at $\upsilon_i$. By Theorem \ref{main theorem}
			we know that $NP_q(L(\rho_E,s))$ lies above 
			\begin{align*} 
			HP(\rho_E)&=\big 
			\{0,0,0,0 \}\sqcup \big 
			\{1,1,1,1 \}
			\sqcup \Big ( \bigsqcup_{i=1}^4 \Big\{\frac{1}{2s_i}, \dots, \frac{2s_i-1}{2s_i}\Big\} \Big ).
			\end{align*} 
			This follows by recognizing that $s_i'=2s_i$ and $\omega_i'=0$. 
			The factorization $L(\rho_E,s)=L(\rho,s)L(\rho^{wild},s)$ corresponds
			to a ``decomposition'' of $HP(\rho_E)$ into two Hodge polygons, one giving a lower
			bound of $NP_q(L(\rho,s))$ and the other for $NP_q(L(\rho^{wild},s))$.
			We have $\omega_i=\frac{a(p-1)}{2}$ for each $i$
			and $\Omega_\rho=2$. This allows us to compute the Hodge polygons as 
			\begin{align*}
				HP(\rho)&= \{0,1\} \sqcup \Big ( \bigsqcup_{i=1}^4\Big \{\frac{1}{2s_i}, \frac{3}{2s_i}, \dots, \frac{2s_i-1}{2s_i}\Big\} \Big ),\\
				HP(\rho^{wild}) &=  \big 
				\{0,0,0 \}\sqcup \big 
				\{1,1,1 \}
				\sqcup \Big ( \bigsqcup_{i=1}^4 \Big\{\frac{1}{s_i}, \dots, \frac{s_i-1}{s_i}\Big\} \Big ),
			\end{align*}
			so that $HP(\rho_E)=HP(\rho)\sqcup HP(\rho^{wild})$. 
			More generally, we will obtain similar decompositions of the Hodge bounds
			as long as $Im(\chi)|p-1$. 
			
		\end{example}
		\subsubsection{Newton polygons of abelian covers of curves}
		Theorem \ref{main theorem}
		also has interesting consequences for Newton polygons of 
		cyclic covers of curves. Let $G=\Z/Np^n\Z$, where $N$ is coprime to
		$p$. Let $f: C \to X$ be a $G$-cover
		ramified over $\tau_1,\dots,\tau_\boldm$. 
		We let $H^1_{cris}(X)$ (resp. $H^1_{cris}(C)$) be the crystalline cohomology
		of $X$ (resp. $C$). For a character $\rho$ of $G$, we let
		$H^1_{cris}(C)^\rho$ be the $\rho$-isotypical subspace for the action of $G$ on $H^1_{cris}(C)$. 
		Let $NP_C$ (resp. $NP_X$ and $NP_C^\rho$) denote the $q$-adic Newton polygon
		of $\det(1-s\text{F}|H^1_{cris}(C))$ (resp. $\det(1-s\text{F}|H^1_{cris}(X))$ and 
		$\det(1-s\text{F}|H^1_{cris}(C)^\rho)$).
		We are interested in the following question: to what extent can
		we determine $NP_C$ from $NP_X$ and the ramification invariants of $f$? 
		The most basic result is the Riemann-Hurwitz formula, which determines
		the dimension of $H^1_{cris}(C)$ from $H^1_{cris}(X)$ and the ramification invariants.
		When $N=1$, there is also the Deuring-Shafarevich
		formula (see \cite{Crew-p-covers}), which determines the number of slope zero segments in $NP_C$. In general, however, a 
		precise
		formula for the slopes of $NP_C$ seems impossible. Instead,
		the best we may hope for are estimates. To connect
		this problem to Theorem \ref{main theorem}, recall the
		decomposition:
		\begin{align} \label{decomposition of H1}
		\det(1-s\text{F}|H^1_{cris}(C))&= \det(1-s\text{F}|H^1_{cris}(X)) \prod_{\rho} \det(1-s\text{F}|H^1_{cris}(C)^\rho),
		\end{align}
		where $\rho$ varies over the nontrivial characters 
		$\Z/Np^n\Z \to \mathbb{C}^\times$. By the Lefschetz trace formula we know
		$L(\rho,s)=\det(1-s\text{F}|H^1_{cris}(C)^\rho)$. Thus, Theorem \ref{main theorem} gives
		lower bounds for $NP_C$ using \eqref{decomposition of H1}.

		Consider the case where $N=1$, so that $G=\Z/p^n\Z$. Let
		$r_i$ be the ramification index of a point of $C$ above $\tau_i$ and define \[\Omega = \sum_{i=1}^\boldm p^{n-r_i}(p^{r_i}-1).\] For
		$j=1,\dots,n$, let $C_j$ be the cover of $X$ corresponding to the subgroup
		$p^{n-j}\Z / p^n\Z \subset G$. Fix a point $x_i(j)\in C_j$ above $\tau_i$; this gives a local field extension
		of $\mathbb{F}_q((t_{\tau_i}))$. 
		We let $s_{\tau_i}(j)$ denote the
		largest upper numbering ramification break of this extension. 
		\begin{corollary} \label{corollary: ASW Newton Polygon}
			The Newton polygon $NP_C$ lies above the polygon 
			whose slopes are the multi-set: 
			\[ NP_X \sqcup  
			\Bigg\{\underbrace{0,\dots,0}_{(p^n-1)(g-1)+\Omega }, 
			\underbrace{1,\dots,1}_{(p^n-1)(g-1)+\Omega }\Bigg\} \sqcup \Bigg(\bigsqcup_{i=1}^\boldm \bigsqcup_{j=1}^n
			p^{j-1}(p-1)\times \Bigg\{\frac{1}{s_{\tau_i}(j)}, \dots, \frac{s_{\tau_i}(j)-1}{s_{\tau_i}(j)} \Bigg\}\Bigg),\]
			where we take $\{\frac{1}{s_{\tau_i}(j)}, \dots, \frac{s_{\tau_i}(j)-1}{s_{\tau_i}(j)} \}$
			to be the empty set when $s_{\tau_i}(j)=0$.
		\end{corollary}
		\begin{remark}
			When $N >1$, we can obtain a complicated bound for $NP_C$ from Theorem \ref{main theorem} and \eqref{decomposition of H1}.  Alternatively, we can replace $X$ with the intermediate curve $X^{tame}$ satisfying
			$Gal(C/X^{tame})=\Z/p^n\Z$ and then apply Corollary \ref{corollary: ASW Newton Polygon} to
			the cover $C \to X^{tame}$ to obtain a bound. Both bounds are the same. 
		\end{remark}
		\subsection{Outline of proof}
		The classical approaches to studying $p$-adic properties of exponential sums on tori
		no longer work when one considers more general curves.
		Instead, we have to expand on the methods developed in earlier work of
		the author on exponential sums on curves (see \cite{kramermiller-padic}). 
		We use the Monsky trace formula (see \S \ref{subsection: MW trace formula}).
		This trace formula allows us to compute $L(\rho,s)$ by studying Fredholm determinants of certain operators.
		More precisely, let $V=X-\{\tau_1,\dots,\tau_\boldm \}$ and let $\overline{B}$ be the
		coordinate ring of $V$. Let $L$ be a finite extension of $\Q_p$
		whose residue field is $\mathbb{F}_q$ such that the image of $\rho$ is contained
		in $L^\times$. Let $B^\dagger$ be the ring of
		\emph{integral overconvergent functions} on a formal lifting of $\overline{B}$ 
		over $\mathcal{O}_L$ (see  \S \ref{section: global bounds}). For example, if $V=\mathbb{A}^1$, then $B^\dagger=\mathcal{O}_L\ang{t}^\dagger$ (i.e. $B^\dagger$ is the ring of power series
		with integral coefficients that converge beyond the closed unit disc). 
		Choose an endomorphism $\sigma: B^\dagger \to B^\dagger$ that
		lifts the $q$-power Frobenius of $\overline{B}$. 
		Using $\sigma$ we define an operator $U_q: B^\dagger\to 
		B^\dagger$, which is the composition
		of a trace map $Tr:B^\dagger\to \sigma(B^\dagger)$ with 
		$\frac{1}{q}\sigma^{-1}$.

		The Galois representation $\rho$ corresponds to
		a unit-root overconvergent $F$-crystal of rank one. 
		This is a 
		$B^\dagger$-module $M=B^\dagger e_0$
		and a $B^\dagger$-linear isomorphism $\varphi: M \otimes_{\sigma} 
		B^\dagger \to M$. 
		Note that this $F$-crystal is determined 
		entirely by $\alpha \in B^\dagger$ satisfying
		$\varphi(e_0 \otimes 1) = \alpha e_0$. We refer to $\alpha$
		as the Frobenius structure of $M$.
		In our specific setup (see \S \ref{section: global bounds}),
		the Monsky trace formula can be written as follows:
		\begin{align*}
		L(\rho,s) &= \frac{\det(1-sU_q \circ \alpha |B^\dagger)}
		{\det(1-sqU_q \circ \alpha |B^\dagger)},
		\end{align*}
		where we regard $\alpha$ as the ``multiplication by $\alpha$'' map on $B^\dagger$.
		Thus, we need to understand the operator $U_q \circ \alpha$. Let us
		outline how we study this operator.
				
		\paragraph{Lifting the Frobenius endomorphism.} 
		Both $U_q$ and $\alpha$ depend on the choice of Frobenius endomorphism $\sigma$.
		When $V=\mathbb{G}_m$, the ring $B^\dagger$ is $\mathcal{O}_L\langle t \rangle^\dagger$, and the natural choice 
		for $\sigma$ sends $t$ to $t^q$. However, no such natural
		choice exists for higher genus curves. Our approach is to pick a convenient mapping
		$\eta:X \to \mathbb{P}^1$ and then pull back the Frobenius $t \mapsto t^q$ along $\eta$.
		We take $\eta$ to be a tamely ramified map that is \'etale outside of $\{0,1,\infty\}$. 
		We may further assume that $\eta(\tau_i) \in \{0,\infty\}$ and the ramification index
		of every point in $\eta^{-1}(1)$ is $p-1$ (see Lemma \ref{lemma: map to p1}). 
		This leaves us with two types of local Frobenius endomorphisms. For	$Q \in X$
		with $\eta(Q)\in \{0,\infty\}$, we may take the local parameter at $Q$
		to look like $u_Q=t^{\pm\frac{1}{e_Q}}$, where $e_Q$ is the ramification index at $Q$. In particular, the Frobenius endomorphism
		sends $u_Q \mapsto u_Q^q$. If $\eta(Q)=1$, we take the local parameter
		to look like $u_Q=\sqrt[p-1]{t-1}$. Thus, the Frobenius endomorphism
		sends $u_Q \mapsto \sqrt[p-1]{(u_Q^{p-1}+1)^{p}-1}$. In \S \ref{section: local Up operators}
		we study $U_q$ for both types of local Frobenius endomorphisms and in
		\S \ref{subsection: some local Frobenius structures} we study the local
		versions of the Frobenius structure $\alpha$.  
		
		\paragraph{The problem of $a$-th roots of $U_q\circ \alpha$.}
		To obtain the correct estimates of $\det(1-sU_q \circ \alpha |B^\dagger)$,
		it is necessary to work with an $a$-th root of $U_q\circ \alpha$. That is,
		we need an element $\alpha_0 \in B^\dagger$ and a $U_p$ operator (this is analogous
		to the $U_q$ operator, but for liftings of the $p$-power endomorphism) such that
		$(U_p\circ \alpha_0)^a=U_q\circ \alpha$. However, this $a$-th root
		is only guaranteed to exist if the order of $Im(\chi)$ divides $p-1$ (see \S \ref{subsection: local rank one crystals}). This presents a major technical obstacle. The solution is to consider $  \rho^{wild}\otimes\bigoplus\limits_{j=0}^{a-1}\chi^{\otimes p^j}$, which is a restriction of scalars of $\rho$. The $L$-functions of each summand are Galois conjugate, and thus have
		the same Newton polygon. We can then study an operator $U_p \circ N$, where
		$N$ is the Frobenius structure of the $F$-crystal associated to $  \rho^{wild}\otimes\bigoplus\limits_{j=0}^{a-1}\chi^{\otimes p^j}$.
		This is similar to the idea used in
		Adolphson and Sperber's study of twisted exponential sums on tori (see \cite{Adolphson-Sperber-twisted1}). 
		They present it in an ad-hoc manner, but the underlying idea is to study
		$  \rho^{wild}\otimes\bigoplus\limits_{j=0}^{a-1}\chi^{\otimes p^j}$ in lieu of $\rho$.

		\paragraph{Global to local computations.}
		When $V$ is $\mathbb{G}_m$ or $\mathbb{A}^1$, the ring $B^\dagger$
		is just $\mathcal{O}_L\ang{t}^\dagger$ or $\mathcal{O}_L\ang{t,t^{-1}}^\dagger$. In both cases,
		it is relatively easy to study operators on $B^\dagger$. 
		The situation is more complex for higher genus curves. Our approach
		to make sense of $B^\dagger$ is to ``expand'' each function
		around the $\tau_i$ (and some other auxiliary points). Namely, let
		$t_i\in B^\dagger$ be a function whose reduction in $\overline{B}$ has a simple zero
		at $\tau_i$. We let $\mathcal{O}_{\mathcal{E}_i^\dagger}$ be the ring of formal Laurent
		series in $t_i$ that converge on an annulus $r<|t_i|_p<1$ (i.e. the bounded Robba ring).
		Any $f \in B^\dagger$ has a Laurent expansion in $t_i$, and
		our overconvergence condition implies this expansion lies in 
		$\mathcal{O}_{\mathcal{E}_i^\dagger}$. We obtain an injection:
		\begin{align} \label{local to global expansion eq1}
		B^\dagger \hookrightarrow  \bigoplus_{i=1}^{\mathbf{m}} \mathcal{O}_{\mathcal{E}_i^\dagger}.
		\end{align}
		The operator $U_p \circ N$ extends to an operator on each
		summand. By carefully
		keeping track of the image of $B^\dagger$, we are able to compute on each summand (see \S \ref{subsection: estimating NP_p}). 
		This
		lets us compute $U_p\circ N$ on the bounded Robba ring, which ostensibly looks
		like a ring of functions on $\mathbb{G}_m$. We are thus able to compute
		$U_p \circ N$ by studying local Frobenius structures and local $U_p$ operators.
		
		\paragraph{Comparing Frobenius structures and $\Omega_\rho$.}
		In \S \ref{subsection: some local Frobenius structures} we
		study the shape of the unit-root $F$-crystal associated
		to $\rho$ when we localize at a ramified point $\tau_i$. 
		We show that the localized unit-root $F$-crystal has a particularly nice Frobenius structure, which depends
		on the ramification datum. However, these well-behaved local Frobenius structures do not
		patch together to give a well-behaved global Frobenius structure. This is a major technical obstacle.
		When comparing local and global
		Frobenius structures we end up having to ``twist'' the image
		of \eqref{local to global expansion eq1}. This process explains the invariant
		$\Omega_\rho$ occurring in Theorem \ref{main theorem}-- it arises by ``averaging'' the local exponents for each $\rho^{wild} \otimes \chi^{\otimes p^i}$. 
		This invariant is essentially absent in the work of Adolphson-Sperber,
		since $\Omega_\rho=1$ if $V=\mathbb{G}_m$. It is also absent in the author's previous work,
		where the local exponents were all zero.

		\subsection{Further remarks}
		Pinning down the exact Newton polygon of a covering
		of a curve, as well the Newton polygon of the isotypical constituents, is
		a fascinating question. A general answer seems impossible, but one can
		certainly hope for results that hold generically. If the genus of $X$
		and the monodromy invariants from \S \ref{subsection: statement of main results} are
		specified, what is the Newton polygon for a generic character? 
		We believe the bound from Theorem \ref{main theorem} should only
		be generically attained if $N|p-1$ and there are some congruence relations
		between $p$ and the Swan conductors. When $\rho$ factors through
		an Artin-Schreier cover, this is known by combining work of the author (see \cite{kramermiller-padic})
		together with work of Booher-Pries (see \cite{Pries-Booher}). The next step would be
		to study the case arising from a cyclic cover whose degree divides $p(p-1)$ (or
		even allowing higher powers of $p$).  
		When $N\nmid p-1$, the bound from Theorem \ref{main theorem} has too many
		slope zero segments. The issue is that a generic tame cyclic cover of degree $N$
		is not ordinary, even if $X$ is ordinary (see \cite{Bouw}). Even when $X=\mathbb{P}^1$,
		the study of Newton polygons
		for tame cyclic covers is already a complicated topic
		(see e.g. \cite{Li-Mantovan-Pries-Tang-families_of_cyclic}). The author plans to return to these questions
		at a later time. It would also be interesting to prove Hodge bounds for
		representations with positive weight. In recent work of Fres\'an-Sabbah-Yu, the authors
		use irregular Hodge theory to study the $p$-adic slopes of symmetric powers of
		Kloosterman sums (see \cite{fresn2018hodge}). Not much is known beyond this case.

		\subsection{Acknowledgments}
		Throughout the course of this work, we have
		benefited greatly from conversations with
		Daqing Wan, Stephen Sperber and Rachel Pries. We also
		thank Jeng-Daw Yu, who asked us if the methods from \cite{kramermiller-padic} could be applied to
		character sums. We would also like to thank the anonymous referee for many helpful suggestions.

	\section{Notation}
	
	\subsection{Conventions} \label{section: convensions}
	
	The following conventions will be used throughout the article.
	We let $\mathbb{F}_q$ be an extension of $\mathbb{F}_p$ with
	$a=[\mathbb{F}_q:\mathbb{F}_p]$. It is enough
	to prove Theorem \ref{main theorem}
	after replacing $q$ with a larger power of $p$. In particular, we increase $q$ throughout the article 
	if it simplifies arguments. We will frequently have families of 
	things indexed by $i=0,\dots, a-1$ (e.g. the $p$-adic digits $e_{Q,i}$ of $\epsilon_{Q}$
	from \S \ref{subsection: statement of main results}). It will be convenient
	to have the indices ``wrap around'' modulo $a$. That is, we take $e_{Q,a}$ to be
	$e_{Q,0}$, $e_{Q,a+1}$ to be $e_{Q,1}$, and so forth.

	Let $L_0$ be the 
	unramified
	extension of $\Q_p$ whose residue field is $\mathbb{F}_q$. Let $E$
	be a finite totally ramified extension of $\Q_p$ of degree $e$ and set 
	$L=E\otimes_{\Q_p} L_0$. Define $\mathcal{O}_L$ (resp. $\mathcal{O}_E$) to be the ring of integers of $L$ (resp. $E$) and let
	$\gothm$ be the maximal ideal of $\mathcal{O}_L$. We let $\pi_\circ$ be a uniformizing element of $E$. Fix $\pi=(-p)^{\frac{1}{p-1}}$ and for any positive rational number 
	$s$ 
	we set 
	$\pi_s=\pi^{\frac{1}{s}}$.  We will assume that $E$ is large enough to contain 
	$\pi_{s_{\tau_i}}$ for each $i=1,\dots, \mathbf{m}$. We also
	assume that $E$ is large enough to contain the image of $\rho^{wild}$ (i.e. $E$ contains
	enough $p$-th power roots of unity).
	Define $\nu$ to be
	the endomorphism $\text{id}\otimes \text{Frob}$ of $L$,
	where $\text{Frob}$ is the $p$-Frobenius automorphism of $L_0$. 
	For any $E$-algebra $R$ and $x \in R$, we obtain an operator $R \to R$
	sending $r \mapsto xr$. By abuse of notation, we will refer to this operator as $x$. Finally,
	for any ring $R$ 
	with
	valuation $v:R \to \mathbb{R}$ and any $x \in R$ with
	$v(x)>0$, we let $v_x(\cdot)$ denote the normalization
	of $v$ satisfying $v_x(x)=1$.

	\subsection{Frobenius endomorphisms} \label{subsection: Frobenius endomorphisms}
	Let $\overline{A}$ be an $\mathbb{F}_q$-algebra,
	let $A$ be an $\mathcal{O}_L$-algebra with $A\otimes_{\mathcal{O}_L} \mathbb{F}_q = \overline{A}$,
	and let $\mathcal{A}=A \otimes_{\mathcal{O}_L} L$. 
	A $p$-Frobenius endomorphism (resp. $q$-Frobenius endomorphism) of $A$ is a ring endomorphism $\nu:A \to A$ (resp. $\sigma: A \to A$)
	that extends the map $\nu$ (resp. $\nu^a=id$) on $\mathcal{O}_L$ defined in \S \ref{section: convensions} and reduces to the $p$-th power map (resp. $q$-th power map) of $\overline{A}$. 
	Note that $\nu$ (resp. $\sigma$) extends to a map $\nu: \mathcal{A} \to \mathcal{A}$ (resp. $\sigma: \mathcal{A} \to \mathcal{A}$), which
	we refer to as a $p$-Frobenius endomorphism (resp. $q$-Frobenius endomorphism) of $\mathcal{A}$. For
	a square matrix $M=(m_{i,j})$ with entries in $\mathcal{A}$ we take $M^{\nu^k}$ to mean the matrix $(m_{i,j}^{\nu^k})$
	and we define $M^{\nu^{a-1} + \dots + \nu + 1}$ by $M^{\nu^{a-1}} \dots M^\nu M$.

	\subsection{Definitions of local rings} \label{subsection: basic definitions}
	We begin by defining some rings
	and modules, which will be used throughout this article.
	Define the $L$-algebras:
	\[  \gls{Amice} = \Bigg\{ \sum_{-\infty}^\infty a_nt^n \Bigg |  
	\begin{array}  {l}
	\text{ We have } a_n\in L, ~\lim\limits_{n\to-\infty} v_p(a_n)=\infty,   \\
	\text{ and}	~	v_p(a_n) \text{ is bounded below.} 
	\end{array}
	\Bigg \},  
	\]\[
	\gls{Bounded Robba} = \Bigg\{ \sum_{-\infty}^\infty a_nt^n  \in 
	\mathcal{E} \Bigg |  
	\begin{array}  {l}
	\text{ There exists $m>0$ such that} \\
	v_p(a_n) \geq -mn \text{ for $n\ll 0$} 
	\end{array}
	\Bigg \}.  
	\]
	We refer to $\mathcal{E}_t$ (resp. $\mathcal{E}^\dagger_t$)
	as the Amice ring (resp. the bounded Robba ring) over $L$ with parameter $t$. We will often omit the $t$ in the subscript if there is no ambiguity. Note that 
	$\mathcal{E}^\dagger$ and 
	$\mathcal{E}$  are local fields with residue field
	$\mathbb{F}_q((t))$. The valuation $v_p$ on $L$ extends to
	the Gauss valuation on each of these fields. 
	We define $\mathcal{O}_{\mathcal{E}}$ (resp. $\mathcal{O}_{\mathcal{E}^\dagger}$)
	to be the subring of $\mathcal{E}$ (resp. $\mathcal{E}^\dagger$) consisting
	of formal Laurent series with coefficients in $\mathcal{O}_L$. 
	Let $u \in \mathcal{O}_{\mathcal{E}^\dagger}$ such that the reduction
	of $u$ in $\mathbb{F}_q((t))$ is a uniformizing element. Then we have
	$\mathcal{E}_u = \mathcal{E}$ (resp. $\mathcal{E}_u^\dagger=\mathcal{E}^\dagger$).
	In particular, we see that $u$ is a different parameter of $\mathcal{E}$.  
	Note that
	if $\nu:\mathcal{E}\to\mathcal{E}$ is any $p$-Frobenius endomorphism,
	we have $\mathcal{E}^{\nu=1}=E$.
	For $m \in \Z$, we define the $L$-vector space of truncated
	Laurent series: 
	\[
	\mathcal{E}^{\leq m} = \Bigg\{ \sum_{-\infty}^\infty a_nt^n  \in 
	\mathcal{E} \Bigg |  
	\begin{array}  {l}
	a_n=0\text{ for all $n>m$} \\
	\end{array}
	\Bigg \}  .
	\]
	The space $\mathcal{E}^{\leq 0}$ is a ring and  $\mathcal{E}^{\leq m}$
	is an $\mathcal{E}^{\leq 0}$-module. There is a natural projection
	$\mathcal{E} \to \mathcal{E}^{\leq m},$
	given by truncating the Laurent series. 
	Finally, we define the following $\mathcal{O}_L$-algebra:
	\[
	\mathcal{O}_{\mathcal{E}(0,r]} = \Bigg\{ \sum_{-\infty}^\infty a_nt^n  
	\in 
	\mathcal{O}_{\mathcal{E}} \Bigg |  
	\lim_{n\to-\infty} v_p(a_{n})  +rn =\infty
	\Bigg \}.  
	\]
	Set $\gls{r-convergent functions}=\mathcal{O}_{\mathcal{E}(0,r]}\otimes_{\mathcal{O}_L} L$. Note that $\mathcal{E}(0,r]$ is the ring of bounded functions on the closed annulus $0<v_p(t)\leq r$. In particular, we have $\mathcal{E}^\dagger= \bigcup\limits_{r>0} \mathcal{E}(0,r]$.
	
	\subsection{Matrix notation} 
	For any
	$c_0,\dots, c_{a-1} \in \mathcal{E}$, we define the $a\times a$ matrices:
	\begin{align*}
		\diag{c_0,\dots,c_{a-1}} &= 
			\begin{pmatrix}
			c_{0} & & \\
			& \ddots & \\
			& & c_{a-1}
			\end{pmatrix}, \\
		\cyc{c_0, \dots, c_{a-1}} & =
			\begin{pmatrix}
			 & c_{0} & & \\
			& & \ddots & \\
			& & & c_{a-2} \\
			 c_{a-1} & &&
			\end{pmatrix},\\
		\tcyc{c_0, \dots, c_{a-1}} & =	 \cyc{c_0, \dots, c_{a-1}}^T.
	\end{align*}

	\section{Global setup} \label{Section global bounds}
	\label{section: global bounds}
	We now introduce the global setup, which closely follows \cite[\S 3]{kramermiller-padic}. We adopt the notation from \S \ref{subsection: statement of main results}. 
	Our main goal is to choose a Frobenius endomorphism on
	a lift of an affine subspace of $X$. We require two things
	from this Frobenius endomorphism. First, we want an endomorphism that
	 behaves reasonably with respect to certain local parameters. Our
	 second requirement is for this Frobenius endomorphism to
	 make the Monsky trace formula satisfy a certain form (see \S \ref{subsection: MW trace formula}).
	 We find this Frobenius endomorphism by bootstrapping from the 
	 standard Frobenius endomorphism on the projective line. 
	
	\subsection{Mapping to $\mathbb{P}^1$}
	\label{subsection: mapping to p1}
	\begin{lemma} \label{lemma: map to p1}
		After increasing $q$, there exists
		a tamely ramified morphism $\eta:X \to \mathbb{P}_{\mathbb{F}_q}^1$,
		ramified only above $0,1$, and $\infty$, such that $\tau_1,\dots, \tau_{\mathbf{m}} \in 
		\eta^{-1}(\{0,\infty\})$ and each $P \in \eta^{-1}(1)$ has ramification index 
		$p-1$.
	\end{lemma}
	
	\begin{proof}
		This is \cite[Lemma 3.1]{kramermiller-padic}.
	\end{proof}

	\subsection{Basic setup} \label{subsection: basic setup}
	Write
	$\mathbb{P}^1_{\mathbb{F}_q}=\Proj(\mathbb{F}_q[x_1,x_2])$
	and let $\bar{t}=\frac{x_1}{x_2}$ be a parameter at $0$. 
	Fix a morphism $\gls{morphism to P1}$ as in Lemma \ref{lemma: map to p1}. 
	For $* \in \{0,1,\infty\}$ we let
	$\{P_{*,1}, \dots, P_{*,r_*}\} = \eta^{-1}(*)$ and set $\gls{W}= 
	\eta^{-1}(\{0,1,\infty\})$. Again, we will increase $q$ so that
	each $P_{*,i}$ is defined over $\mathbb{F}_q$. Fix $Q=P_{*,i} \in W$. We define $\gls{eQ}$ to be 
	the ramification
	index of $Q$ over $*$. From Lemma \ref{lemma: map to p1}, if $*=1$ we have 
	$e_{Q}=p-1$ 
	for 
	$1\leq i \leq r_1$, so that $r_1(p-1)=\deg(\eta)$. Also, by the Riemann-Hurwitz formula
	\begin{align}
	(g-1) + (r_0+r_1 + r_\infty) &= \deg(\eta)-g+1, \label{riemann-hurwitz eq}
	\end{align}
	where
	$g$ denotes the genus of $X$. Let $U=\mathbb{P}^1_{\mathbb{F}_q}-\{0,1,\infty \}$
	and $\gls{V}=X-W$. Then $\eta: V \to U$ is a finite \'etale map of degree $\deg(\eta)$. 
	Let $\overline{B}$ (resp. $\overline{A}$) be the $\mathbb{F}_q$-algebra
	such that $V=\Spec(\overline{B})$ (resp. $U=\Spec(\overline{A})$).
	
	Let $\mathbb{P}^1_{\mathcal{O}_L}$ be the projective line
	over $\Spec(\mathcal{O}_L)$
	and let $\mathbf{P}^1_{\mathcal{O}_L}$ be the formal projective
	line over $\text{Spf}(\mathcal{O}_L)$. Let $t$ be a global parameter
	of $\mathbf{P}^1_{\mathcal{O}_L}$ lifting $\bar{t}$. By the deformation theory of 
	tame
	coverings (see \cite[Theorem 4.3.2]{Grothendieck-Murre-tame_fundamental_groups}) 
	there exists a tame cover
	$\mathbf{X} \to \mathbf{P}^1_{\mathcal{O}_L}$ 
	whose special fiber is $\eta$ and by formal GAGA (see 
	\cite[\href{https://stacks.math.columbia.edu/tag/09ZT}{Tag 
		09ZT}]{stacks-project}) 
	there exists a morphism
	of smooth curves $\mathbb{X} \to \mathbb{P}^1_{\mathcal{O}_L}$
	whose formal completion is $\mathbf{X} \to 
	\mathbf{P}^1_{\mathcal{O}_L}$. 
	
	Define the functions
	$t_{0}=t$, $t_\infty = \frac{1}{t}$ and $t_1=t-1$. Let
	$[*]$ denote the $\mathcal{O}_L$-point of $\mathbb{P}^1_{\mathcal{O}_L}$
	given by $t_*=0$. For $Q =P_{*,i}$, let $[Q]$ be a point of $\eta^{-1}([*])$
	that reduces to $Q$ in the special fiber. Note that such a point exists
	since $Q \in \eta^{-1}(*)$, but it is not necessarily unique. 
	Let
	$\mathbb{U} = \mathbb{P}^1_{\mathcal{O}_L} - \{ [0], [1],[\infty] \}$
	and $\mathbb{V} = \mathbb{X}-\{[R]\}_{R\in W}$. We define $\mathbf{U} = 
	\mathbf{P}^1_{\mathcal{O}_L} - \{ 0,1,\infty \}$
	and $\gls{V formal} = \mathbf{X}-\{R\}_{R\in W}$. Then $\mathbf{U}$ (resp. 
	$\mathbf{V}$) is the formal completion of $\mathbb{U}$ (resp. $\mathbb{V}$). 
	We
	let $\mathcal{U}^{rig}$ (resp. $\gls{V rigid}$) be the rigid analytic fiber 
	of $\mathbf{U}$  (resp. $\mathbf{V}$). 
	Let $\widehat{A}$ (resp. $\widehat{\mathcal{A}}$) be the ring 
	of functions
	$\mathcal{O}_{\mathbf{U}}(\mathbf{U})$ (resp. $\mathcal{O}_{\mathcal{U}^{rig}}(\mathcal{U}^{rig})$) and let $\widehat{B}$
	(resp. $\widehat{\mathcal{B}}$) be the ring of 
	functions
	$\mathcal{O}_{\mathbf{V}}(\mathbf{V})$ (resp. 
	$\mathcal{O}_{\mathcal{V}^{rig}}(\mathcal{V}^{rig})$).

	\subsection{Local parameters and overconvergent rings}
	\label{subsection: expansion around local parameters}
	For $Q=P_{*,i}$, let $w_Q$ be a rational function on $\mathbb{X}$ that has a simple zero
	at $Q$. Let $\mathcal{E}_{*}$ (resp. $\mathcal{E}_{Q}$) be the Amice ring over $L$
	with parameter $t_{*}$ (resp. $w_{Q}$). By expanding functions 
	in terms of the $t_{*}$ and $w_{Q}$, we obtain the following diagrams:
	
	\begin{equation}\label{Local expansion commutative diagram}
	\begin{tikzcd}
	\widehat{B} \arrow[r]& 
	\bigoplus\limits_{Q \in W} \mathcal{O}_{\mathcal{E}_{Q}}  &
	\widehat{\mathcal{B}} \arrow[r] & \bigoplus\limits_{Q \in W} 
	\mathcal{E}_{Q} \\
	\widehat{A} \arrow[r]\arrow[u]& 
	\bigoplus\limits_{*\in\{0,1,\infty\}} \mathcal{O}_{\mathcal{E}_{*}} \arrow[u] 
	&
	\widehat{\mathcal{A}} \arrow[r]\arrow[u] & \bigoplus\limits_{*\in\{0,1,\infty\}} 
	\mathcal{E}_{*}\arrow[u].
	\end{tikzcd}
	\end{equation}
	We let $A^\dagger$ (resp. 
	$\gls{integral overconvergent on V}$) be the subring of
	$\widehat{A}$ (resp. $\widehat{B}$)
	consisting of functions that are overconvergent
	in the tube $]*[$ for each $*\in\{0,1,\infty\}$ (resp. $]Q[$ for all $Q \in W$).
	In particular, $B^\dagger$ fits into the following
	Cartesian diagram:
	\begin{equation}  \label{overconvergent expansion diagram}
	\begin{tikzcd}
	B^\dagger \arrow[d]\arrow[r] & \bigoplus\limits_{Q \in W} \arrow[d] 
	\mathcal{O}_{\mathcal{E}_{Q}}^\dagger \\
	\widehat{B}\arrow[r] & \bigoplus\limits_{Q \in W} \mathcal{O}_{\mathcal{E}_{Q}} 
	.
	\end{tikzcd}
	\end{equation}
	Note that $A^\dagger$ (resp. $B^\dagger$) is the weak completion of $A$ (resp. 
	$B$) 
	in the sense of \cite[\S 2]{Monsky-Washnitzer-formal_cohomology1}. In particular,
	we have $A^\dagger =\mathcal{O}_L \Big < t,t^{-1}, \frac{1}{t-1} \Big>^\dagger$
	and $B^\dagger$ is a finite \'etale $A^\dagger$-algebra.
	Finally, we define $\mathcal{A}^\dagger$ (resp. $\mathcal{B}^\dagger$)
	to be $A^\dagger \otimes \Q_p$ (resp. $B^\dagger \otimes \Q_p$). 
	Then $\mathcal{A}^\dagger$ (resp. $\gls{overconvergent on V}$) is
	equal to the functions in $\widehat{\mathcal{A}}$ (resp. $\widehat{\mathcal{B}}$) 
	that are overconvergent 
	in the tube $]*[$ for each $*\in\{0,1,\infty\}$ (resp. $]R[$ for all $R \in W$).

	The extension $\mathcal{E}_Q^\dagger/\mathcal{E}_*^\dagger$ is an unramified
	extension of local fields and thus completely
	determined by the residual extension. By our assumption on the tameness of $\eta$, we know that this residual extension is tame and can be written as $\mathbb{F}_q((t_*^{\frac{1}{e_Q}}))/\mathbb{F}_q((t_*))$. 
	Since $\mathcal{O}_{\mathcal{E}_Q^\dagger}$ is Henselian (see \cite[Proposition 3.2]{Matsuda-local_indices}) there exists a parameter $u_Q$ of $\mathcal{E}_Q^{\dagger}$
	such that $u_Q^{e_Q}=t_*$. We remark that $u_Q$ will
	be defined on an annulus inside the disc $]Q[$ and in general
	it will not extend to a function on the whole disc.
	
	We will need to consider functions in $\mathcal{B}^\dagger$ with a precise radius of overconvergence
	in terms of the parameters $u_Q$.
	Let $\gls{radius tuple}=(r_Q)_{Q \in W}$ be a tuple of positive rational numbers. We define $\gls{r converge on V}$ to be the subring of functions in $\mathcal{B}^\dagger$ that overconverge
	in the annulus $0<v_p(u_Q)\leq r_Q$.
	\footnote{Our definition of $\mathcal{B}(0,\mathbf{r}]$ is somewhat
	nonstandard. Typically one measures `overconvergence' with global functions. However, if the $r_Q$ are small enough
	the spaces of functions we obtain are essentially the same.
	For example, consider the modular curve case (i.e. $X=X_0(N)$ and $V$ is the ordinary locus). One typically looks at affinoid spaces $\mathcal{X}_0(N)^{r'}$ of the form $0\leq v_p(E_{p-1})< r'$, where $E_{p-1}$ is
	the weight $p-1$ Eisenstein series (i.e. a lift of the Hasse invariant). If $r'$ is sufficiently small we
	have $\mathcal{O}_{\mathcal{X}_0(N)}(\mathcal{X}_0(N)^{r'})=\mathcal{B}(0,\mathbf{r}']$, where $\mathbf{r}'$ is the tuple with $r'$
	in each coordinate. This follows from the following two facts. First, note that $E_{p-1}$ is a parameter of $\mathcal{E}_Q^\dagger$ for each supersingular point $Q \in X$, since the Hasse invariant has simple zeros at supersingular points.
	Second, if $u \in \mathcal{E}^\dagger_t$ is a parameter and $u \in\mathcal{E}_t(0,r]$, then we have
	$\mathcal{E}_u(0,r_0]=\mathcal{E}_t(0,r_0]$ for any $r_0<r$.}
	More precisely, $\mathcal{B}(0,\mathbf{r}]$
	fits into the following Cartesian diagram
	\begin{equation} \label{r-convergent expansion diagram}
	\begin{tikzcd}
	\mathcal{B}(0,\mathbf{r}] \arrow[d]\arrow[r] & \bigoplus\limits_{Q \in W} \arrow[d] 
	\mathcal{E}_{Q}(0,r_Q] \\
	\widehat{\mathcal{B}}\arrow[r] & \bigoplus\limits_{Q \in W} \mathcal{E}_{Q}
	\end{tikzcd}.
	\end{equation}
	Note that $\mathcal{B}^\dagger$ is the union over all $\mathcal{B}(0,\mathbf{r}]$.

	\subsection{Global Frobenius and $U_p$ operators} \label{subsection Global Frobenius and Up}
	Let $\nu:\mathcal{A}^\dagger
	\to \mathcal{A}^\dagger$ be the $p$-Frobenius endomorphism that restricts to $\nu$ on $L$ 
	and sends $t$ to $t^p$.
	Let $\sigma=\nu^a$. For $* \in \{0,1,\infty\}$,
	we may extend
	$\nu$ to a $p$-Frobenius endomorphism of $\mathcal{E}_{*}^\dagger$,
	which we refer to as $\nu_{*}$. 
	In terms of the parameters $t_*$, these endomorphisms are given
	as follows:
	\[ t_0 \mapsto t_0^p,~~~t_\infty \mapsto t_\infty^p, \text{ and}~~~t_1\mapsto 
	(t_1+1)^p-1.\]
	
	Since the map $\widehat{A} \to
	\widehat{B}$ is \'etale and both rings are $p$-adically
	complete, we may extend both
	$\sigma$ and $\nu$ to maps 
	$\sigma,\nu:\widehat{B}\to\widehat{B}$. 
	This extends to a $p$-Frobenius endomorphisms $\nu_{Q}$ of
	$\mathcal{E}_{Q}$, which make the diagrams in \eqref{Local expansion 
		commutative diagram} $p$-Frobenius equivariant. Furthermore, since 
	$\nu_{Q}$ 
	extends
	$\nu_{*}$, we know that $\nu_{Q}$ induces 
	a $p$-Frobenius endomorphism of $\mathcal{E}_{Q}^\dagger$. 	It follows from \eqref{overconvergent expansion diagram}
	that $\sigma$ and $\nu$ restrict to
	endomorphisms $\sigma, \nu: \mathcal{B}^\dagger \to \mathcal{B}^\dagger$.
	The $p$-Frobenius endomorphisms 
	$\nu_{Q}$ can be described as follows:
	\begin{enumerate}
		\item When $*$ is $0$ or $\infty$, have $u_{Q}^{\nu_Q}=u_{Q}^p$, since
		$t_*^{\nu_*}=t_*^p$ and $u_{Q}^{e_Q}=t_*$. 
		\item When $*=1$, we have $u_{Q}^{\nu_Q} = \sqrt[p-1]{(u_Q^{p-1}+1)^p-1}$,
		since $t_1^{\nu_1}= (t_1+1)^p-1$ and $u_Q^{p-1}=t_1$. 
	\end{enumerate}
	Following \cite[\S 3]{van_der_Put-MW_cohomology}, there is a trace map
	$Tr_0: \mathcal{B}^\dagger \to \nu(\mathcal{B}^\dagger)$ (resp. $Tr: 
	\mathcal{B}^\dagger \to \sigma(\mathcal{B}^\dagger)$).
	We may define the $U_p$ operator on 
	$\mathcal{B}^\dagger$:
	\begin{align*}
	U_p: \mathcal{B}^\dagger &\to \mathcal{B}^\dagger \\
	x &\mapsto \frac{1}{p} \nu^{-1}(Tr_0(x)).
	\end{align*}
	Similarly, we define $U_q=\frac{1}{q}\sigma^{-1}\circ Tr$, so that
	$U_p^a=U_q$. Note that $U_p$ is $E$-linear and $U_q$ is $L$-linear.
	Both $U_p$ and $U_q$ extend to operators on $\mathcal{E}_Q^\dagger$.

	\section{Local $U_p$ operators}
	\label{section: local Up operators}
	Let $\nu$ be a $p$-Frobenius endomorphism of $\mathcal{E}^\dagger$ (see \S \ref{subsection: Frobenius endomorphisms}). We define
	$U_p$ to be the map:
	\begin{align*}
	\frac{1}{p}\nu^{-1}\circ \text{Tr}_{\mathcal{E}^\dagger / 
		\nu(\mathcal{E}^\dagger)}: \mathcal{E}^\dagger \to \mathcal{E}^\dagger. 
	\end{align*}
	Note that $U_p$ is $\nu^{-1}$-semi-linear (i.e. $U_p(y^\nu x)=yU_p(x)$ for 
	all $y\in 
	\mathcal{E}^\dagger$).
	In this section we will study $U_p$ for the $p$-Frobenius
	endomorphisms of $\mathcal{E}^\dagger$ appearing \S \ref{subsection Global Frobenius and Up}. 
		\subsection{Type 1: $t \mapsto t^p$}

First consider the $p$-Frobenius endomorphism $\nu:\mathcal{E}^\dagger \to \mathcal{E}^\dagger$ 
sending $t$ to $t^p$. We see that $U_p(t^i)=0$ if $p\nmid i$ and $U_p(t^i)=t^{\frac{i}{p}}$
if $p \mid i$.
Thus, for $s> 0$ we have:
\begin{align}
U_p(\mathcal{O}_\mathcal{E}^s) \subset \mathcal{O}_\mathcal{E}^{\frac{s}{p}}, ~~\text{ and } ~~
U_p(\mathcal{O}_L\llbracket  \pi_s t^{-1} \rrbracket)\subset \mathcal{O}_L\llbracket  \pi_s^p t^{-1} \rrbracket  . 
\label{basic p-frob 
	growth}
\end{align}
	\subsubsection{Local estimates} \label{subsubsection: type one frob local estimates}
	Let $R=(s,\mathbf{e},\epsilon,\omega)$ be a ramification datum and let $e_0,\dots,e_{a-1}$ be the $p$-adic digits of $\epsilon$
	as in \S \ref{subsection: statement of main results}.
	For
	$j=0,\dots,a-1$ we define
	\begin{align*}
	q(\mathbf{e},j) &= - \sum_{i=0}^{a-1} (i+1) e_{i+j}.
	\end{align*}
	Note that
	\begin{align}
	q(\mathbf{e},j) - q(\mathbf{e},j+1) &= ae_j - \omega . \label{eq: differences between 
		the qs}
	\end{align}
	
	\noindent Let $t_i^n \in  \bigoplus\limits_{j=0}^{a-1} \mathcal{E}^\dagger$ 
	denote the element that has $t^n$ in the $i$-th coordinate and zero in the other coordinates.
	We then define the spaces
	\begin{align*}
		\mathcal{D}_{\mathbf{e},s}^{(j)}&= 
		\pi_{as}^{q(\mathbf{e},j)} \pi_s^pt_j^{-1} \mathcal{O}_L\llbracket  \pi_s^pt_j^{-1} \rrbracket \oplus \mathcal{O}_L \llbracket  t_j \rrbracket, \\
		\gls{special local for type 1} &= \bigoplus_{j=0}^{a-1} \mathcal{D}_{\mathbf{e},s}^{(j)} 
		\subset \bigoplus_{j=0}^{a-1} \mathcal{E}^\dagger.
	\end{align*}
	We know $-q(\mathbf{e},i) \leq a(p-1)$, which implies $\pi_{as}^{q(\mathbf{e},j)} \pi_s^p \in \mathcal{O}_L$. In particular, 
	\begin{align*}
		\mathcal{D}_{\mathbf{e},s} &\subset \bigoplus_{j=0}^{a-1} \mathcal{O}_{\mathcal{E}^\dagger}.
	\end{align*}
	\begin{proposition} \label{proposition: type 1 Up estimates}
		Let $\nu$ be the $p$-Frobenius endomorphism that
		sends $t \mapsto t^p$. 
		Let $\alpha \in \mathcal{O}_L\llbracket  \pi_s t^{-1} \rrbracket$
		and set $A=\tcyc{\alpha t^{-e_0}, \dots,\alpha  t^{-e_{a-1}}}$. 
		Then:
		\begin{align}
		U_p \circ A(\mathcal{D}_{\mathbf{e},s}) &\subset \mathcal{D}_{\mathbf{e},s}, \label{proposition: type 1 estimate eq 1}\\
		U_p \circ A(\pi_{as}^{q(\mathbf{e},j)} \pi_s^{np} t_j^{-n}) &\subset \pi_s^{n(p-1)} \pi_{as}^{-\omega}\mathcal{D}_{\mathbf{e},s}, \label{proposition: type 1 estimate eq 2}
		\end{align}
		 for $n\geq 1$ and $0\leq j \leq a-1$.
	\end{proposition}
	
	\begin{proof}
		For $n\geq 1$ we have $A(t_j^{-n})=\alpha t_{j+1}^{-n-e_{j+1}}$. Then from \eqref{eq: differences between the qs} we have:
		\begin{align*}
			A(\pi_{as}^{q(\mathbf{e},j)}\pi_s^{pn} t_j^{-n}) &=
			\pi_{as}^{q(\mathbf{e},j)} \pi_s^{pn} t_{j+1}^{-n-e_{j+1}} \alpha \\
			&= \pi_{as}^{q(\mathbf{e},j+1)} \pi_s^{n(p-1)}\pi_{as}^{-\omega} \cdot (\pi_s^{n + e_{j+1}}t_{j+1}^{-n-e_{j+1}}\alpha).
		\end{align*}
		Note that $\pi_s^{n + e_{j+1}}t_{j+1}^{-n-e_{j+1}}\alpha \in \mathcal{O}_L\llbracket  \pi_s t_{j+1}^{-1} \rrbracket$. Then \eqref{proposition: type 1 estimate eq 2} follows from \eqref{basic p-frob growth}.
		To prove \eqref{proposition: type 1 estimate eq 1}, we need to make sure
		$U_p \circ A(t_j^n) \in \mathcal{D}_{\mathbf{e},s}$ for $n\geq 0$, which
		can be done by a similar argument.
	\end{proof}
	
	\subsection{Type 2: $t \mapsto \sqrt[p-1]{(t^{p-1}+1)^p-1}$}
	\label{subsection: type 2 frobenius}
	Next, consider the $p$-Frobenius endomorphism $\nu:\mathcal{E}^\dagger \to \mathcal{E}^\dagger$ that
	sends
	$t$ to $\sqrt[p-1]{(t^{p-1}+1)^p-1}$.
	Define the following sequence of numbers
	\begin{align*}
	b(n) & =\begin{cases}
		\Big \lfloor \frac{-n-1}{p-1} \Big \rfloor & n\leq -1 \\
		0 & n \geq 0
	\end{cases}. 
	\end{align*}
	We then define the space
	\begin{align}\label{definition of A}
		\gls{special local for type 2} &= \prod_{n \in \Z} p^{b(n)}t^{n}\mathcal{O}_L,
	\end{align}
	which we regard as a sub-$\mathcal{O}_L$-module of $\mathcal{O}_{\mathcal{E}^\dagger}$.
	\begin{proposition} \label{proposition: type 2 Frobenius}
		Let $\nu$ be the $p$-Frobenius endomorphism of $\mathcal{E}^\dagger$ that sends
		$t$ to $\sqrt[p-1]{(t^{p-1}+1)^p-1}$. For all $n\in \Z_{\geq 0}$ and $0\leq k \leq p-1$, we have
		\begin{align*}
		U_p(p^{b(-k-np)}t^{-k-np}) &\in p^n \mathcal{D}, \\
		U_p(\mathcal{D}) &\subset \mathcal{D}. 
		\end{align*}
	\end{proposition}
	\begin{proof}
		See \cite[Proposition 4.4]{kramermiller-padic}. 
	\end{proof}

	\section{Unit-root $F$-crystals}
	\label{section: Local unit-root $F$-isocrystals}	
	
	\subsection{$F$-crystals and $p$-adic representations}
	\label{subsection: local rank one crystals}
	For this subsection, we let $\overline{S}$ be either $\Spec(\mathbb{F}_q((t)))$ or a smooth,
	irreducible affine $\mathbb{F}_q$-scheme $\Spec(\overline{R})$. We let $S=\Spec(R)$
	be a flat $\mathcal{O}_L$-scheme whose special fiber is $\overline{S}$ and assume that $R$ is $p$-adically complete (e.g, if $\overline{S}=\Spec(\mathbb{F}_q((t)))$ then we may take
	$R=\mathcal{O}_\mathcal{E}$).
	Fix a $p$-Frobenius endomorphism $\nu$ on $R$ (as in \S \ref{subsection: Frobenius endomorphisms}). Then $\sigma=\nu^a$
	is a $q$-Frobenius endomorphism. 
	\begin{definition}
		A $\varphi$-module for $\sigma$ over $R$ is a locally
		free $R$-module $M$
		equipped with a $\sigma$-semilinear endomorphism $\varphi: M \to M$. 
		That is,
		we have $\varphi(cm)=\sigma(c)\varphi(m)$ for $c\in R$.  
	\end{definition}	
	\begin{definition} \label{varphi module definition}
		A unit-root $F$-crystal $M$ over $\overline{S}$ is
		a $\varphi$-module such that $\sigma^* \varphi: R \otimes_{\sigma} M \to M$
		is an isomorphism. The rank of $M$ is defined as the rank of the underlying
		$R$-module. 
	\end{definition}

	\begin{theorem}[Katz, see \S 4 in \cite{Katz-p-adic_properties}] 
		\label{mod p riemann-hurwitz}
		There is an equivalence of categories
		\begin{align*}
		\{\text{rank $d$ unit-root $F$-crystals over $\overline{S}$}\} 
		&\longleftrightarrow \{\text{continuous representations} ~ \psi:\pi^{et}_1(\overline{S}) \to GL_d(\mathcal{O}_L)\}.
		\end{align*}
	\end{theorem}	

	 Let us describe a certain case of this correspondence. Let
	 $\overline{S}_1 \to \overline{S}$ be a finite \'etale cover and assume that
	 $\psi$ comes from a map $\psi_0:Gal(\overline{S}_1/\overline{S})\to GL_d(\mathcal{O}_E)$.
	 This cover deforms into a finite \'etale map of affine schemes
	$S_1=\Spec(R_1) \to S$. Both $\nu$ and $\sigma$ extend to $R_1$ and commute with
	the action of $Gal(\overline{S}_1/\overline{S})$ (see e.g. \cite[\S 2.6]{Tsuzuki-finite_local_monodromy}). Let 
	$V_0$ be a free $\mathcal{O}_E$-module of rank $d$ on which $Gal(\overline{S}_1/\overline{S})$ acts on via $\psi_0$ and let $V=V_0 \otimes_{\mathcal{O}_E}\mathcal{O}_L$.  The unit-root $F$-crystal associated to $\psi$ is $M_\psi=(R_1 \otimes_{\mathcal{O}_L} V)^{Gal(\overline{S}_1/\overline{S})}$ with $\varphi=\sigma\otimes_{\mathcal{O}_L} id$. 
	There is a map
	\[ (S_1 \otimes_{\mathcal{O}_{E}} V_0) \to (S_1 
	\otimes_{\mathcal{O}_L} V),\]
	which is Galois equivariant. In particular, the map $\varphi$ has an $a$-th root $\varphi_0=\nu \otimes_{\mathcal{O}_{E}} id$.
	
	Now make the additional assumption that $M_\psi$ is free as an $R$-module.
	Let $e_1,\dots, e_d$ be a basis of $M_\psi$ as an $R$-module and let
	$\mathbf{e}=[e_1,\dots,e_d]$. Then $\varphi(\mathbf{e})=\alpha \mathbf{e}$ (resp. $\varphi_0(\mathbf{e})=\alpha_0\mathbf{e}$), where $\alpha,\alpha_0  \in GL_d(R)$. 
	We refer to the matrix $\alpha$ (resp. $\alpha_0$) as a \emph{Frobenius structure}
	(resp. $p$-\emph{Frobenius structure}) of $M$ and 
	to the matrix $\alpha^T$ (resp. $\alpha_0^T$) as a \emph{dual Frobenius structure} 
	(resp. \emph{dual }$p$\emph{-Frobenius structure}) of $M$. We have the relation 
	$\alpha^T=(\alpha_0^T)^{\nu^{a-1}+ \dots + \nu + 1}$ (recall from \S \ref{subsection: Frobenius endomorphisms}
	that $(\alpha_0^T)^{\nu^{a-1}+ \dots + \nu + 1} = (\alpha_0^T)^{\nu^{a-1}} \dots (\alpha_0^T)^\nu \alpha_0^T$).
	If $\mathbf{e}'=b\mathbf{e}$ and
	$\varphi(\mathbf{e}')=\alpha'\mathbf{e}'$  (resp. $\varphi(e_1)=\alpha_0'e_1$) with $\alpha',\alpha_0',b \in GL_d(R)$, then we have $(\alpha')^T=(b^\sigma)^{T}\alpha^T (b^{-1})^T$ (resp. $(\alpha_0')^T=(b^{\nu})^T\alpha_0^T(b^{-1})^T$).
	In particular, a dual Frobenius structure (resp. dual $p$-Frobenius structure)
	of $M$ is unique up to $\sigma$-skew-conjugation (resp. $\nu$-skew-conjugation) by elements of $GL_d(R)$. We remark that if $M_\psi$ has rank one, then $p$-Frobenius structures (resp. Frobenius structures) 
	are also dual $p$-Frobenius structures (resp. dual Frobenius structures).

	\subsection{Local Frobenius structures}
	\label{subsection: some local Frobenius structures}
	We now restrict ourselves to the case where $\overline{S}=\Spec(\mathbb{F}_q((t)))$.
	In particular, unit-root $F$-crystals over $\overline{S}$ correspond to representations of
	$G_{\mathbb{F}_q((t))}$, the absolute Galois group of $\mathbb{F}_q((t))$. Note that
	since $\mathcal{O}_{\mathcal{E}}$ is a local ring, all locally
	free modules are free.
	
	\subsubsection{Unramified Artin-Schreier-Witt characters}
	\begin{proposition} \label{proposition: unramified character proposition}
		Let $\nu$ be any $p$-Frobenius endomorphism of $\mathcal{O}_\mathcal{E}$
		and let $\sigma=\nu^a$. 
		Let $\psi:G_{\mathbb{F}_q((t))} \to \mathcal{O}_L^\times$ be a continuous character and let $M_\psi$ be the
		corresponding unit-root $F$-crystal. Assume that $Im(\psi)\cong \Z/p^n\Z$
		and that $\psi$ is unramified. Then there exists a $p$-Frobenius structure
		$\alpha_0$ of $M_\psi$ with $\alpha_0 \in 1+\gothm$ (recall $\gothm$ is the
		maximal ideal of $\mathcal{O}_L$). Furthermore, if $c \in 1+\gothm\mathcal{O}_\mathcal{E}$ is another $p$-Frobenius structure
		of $M_\psi$, there exists $b\in 1+\gothm\mathcal{O}_\mathcal{E}$ with $\alpha_0=\frac{b^\nu}{b}c$.
	\end{proposition}

	\begin{proof}
		This is essentially the same as \cite[Proposition 5.4]{kramermiller-padic}.
	\end{proof}

	\subsubsection{Wild Artin-Schreier-Witt characters}
	\label{subsection: frobenius structures for asw extensions}
	A global version over $\mathbb{G}_m$ of the following result is commonplace in the literature (see e.g. \cite[\S 4.1]{Wan-variationNP} for the exponential sum situation
	or \cite{Liu-Wei-witt-coverings}). However, to the best of our knowledge,
	the local version presented below does not appear anywhere. 
	
	\begin{proposition} \label{theorem: ASW frobenius structure}
		Let $\nu$ be the $p$-Frobenius endomorphism of $\mathcal{O}_{\mathcal{E}}$ sending $t$ to $t^p$ and
		let $\sigma=\nu^{a}$. Let $\psi:G_{\mathbb{F}_q((t))} \to \mathcal{O}_L^\times$ be a continuous character and let $M_\psi$ be the
		corresponding unit-root $F$-crystal. Assume that $Im(\psi)\cong \Z/p^n\Z$. Let 
		$K$ 
		be the fixed field of $\ker(\psi)$ and let $s$ be the Swan conductor
		of $\psi$. We assume that $\pi_s \in \mathcal{O}_E$. Then there exists a $p$-Frobenius 
		structure $E_r$ of 
		$\psi$ such that $E_r \in \mathcal{O}_L\llbracket  \pi_s t^{-1} \rrbracket$ and $E_r \equiv 1 \mod \gothm$. Furthermore,
		if $c\in 1+\gothm \mathcal{O}_\mathcal{E}$ is another $p$-Frobenius structure,
		there exists $b\in 1+\gothm \mathcal{O}_\mathcal{E}$ with $E_r=\frac{b^\nu}{b}c$.
	\end{proposition}

	\begin{proof}
		The extension of $K/\mathbb{F}_q((t))$ corresponds to an equivalence class of 
		$W_n(\mathbb{F}_q((t)))/(\textbf{Fr}-1)W_n(\mathbb{F}_q((t)))$ (here $W_n(\mathbb{F}_q((t)))$ is
		the $n$-th truncated Witt vectors and $\textbf{Fr}$ is the Frobenius map). Following \cite[Proposition 3.3]{Kosters-Wan}, we may 
		represent this 
		equivalence class with
		\begin{align*}
		r(t) &= \sum_{i=0}^{n-1} \sum_{j=0}^{s_i} [r_{i,j}t^{-j}] p^i, \text{ with $r_{i,s_i}\neq 0$ and,} \\
		s&= \min_{i=0}^{n-1} \{p^{n-i}s_i\},
		\end{align*}
		where $r_{i,j} \in \mathbb{F}_q$. Since
		$r(t) \in W_n(\mathbb{F}_q[t^{-1}])$,
		the extension $K/\mathbb{F}_q((t))$ extends to  finite \'etale $\mathbb{F}_q[t^{-1}]$-algebra
		$B$ that fits into a commutative diagram:
		\begin{equation*} 
		\begin{tikzcd}
		\Spec(K) \arrow[d]\arrow[r] & \Spec(B)\arrow[d] \\
		\Spec(\mathbb{F}_q((t)))\arrow[r] & \mathbb{P}^1-\{0\}=\Spec(\mathbb{F}_q[t^{-1}])
		.
		\end{tikzcd}
		\end{equation*}
		In particular, $\psi$ extends to a representation $\psi^{ext}:Gal(B/\mathbb{F}_q[t^{-1}]) \to \mathcal{O}_L^\times$. This extension is
		uniquely defined by the following property: for $k\geq 1$ and $x \in \mathbb{P}^1(\mathbb{F}_{q^k})- \{0\}$ we have
		\begin{align} \label{equation: trace of ASW Frobenius}
			\psi^{ext}(Frob_x)&= \zeta_{p^n}^{Tr_{W_n(\mathbb{F}_{q^k})/W_n(\mathbb{F}_p)}(r([x]))},
		\end{align}
		where $[x]$ denotes the Teichmuller lift of $x$ in $W_n(\mathbb{F}_{q^k})$
		and $\zeta_{p^n}$ is a primitive $p^n$-th root of unity.
		Let $\mathcal{O}_L\ang{t^{-1}}\subset \mathcal{O}_\mathcal{E}$ be the Tate algebra
		in $t^{-1}$ with coefficients in $\mathcal{O}_L$. Note that $\nu$ restricts to a
		$p$-Frobenius endomorphism of $\mathcal{O}_L\ang{t^{-1}}$. All projective modules
		over $\mathcal{O}_L\ang{t^{-1}}$ are free, so that $M_{\psi^{ext}}$ is isomorphic to $\mathcal{O}_L\ang{t^{-1}}$ as an $\mathcal{O}_L\ang{t^{-1}}$-module. We see that 
		$M_{\psi}=M_{\psi^{ext}}\otimes_{\mathcal{O}_{L}\ang{t^{-1}}}\mathcal{O}_{\mathcal{E}}$.
		In particular, any $p$-Frobenius structure of $M_{\psi^{ext}}$ is a $p$-Frobenius
		structure of $M_{\psi}$.
		
		A series $\alpha_0 \in \mathcal{O}_L\ang{t^{-1}}$ is a $p$-Frobenius structure for 
		$M_{\psi^{ext}}$ if for every $x \in \mathbb{P}^1(\mathbb{F}_{q^k})- \{0\}$
		we have
		\begin{align*}
		\prod_{i=0}^{ak-1} \alpha_0([x])^{\nu^{i}} &= \psi^{ext}(Frob_x).
		\end{align*}
		We let $E(x)$ denote the Artin-Hasse exponential and let $\gamma_i$
		be an element of $\Z_p[\zeta_{p^n}]$ with $E(\gamma_n)=\zeta_{p^n}^{p^{n-i}}$. Note that
		$v_p(\gamma_i)=\frac{1}{p^{i-1}(p-1)}$. 
		Thus, from \eqref{equation: trace of ASW Frobenius} we see that
		\begin{align*}
		E_r &= \prod_{i=0}^{n-1} \prod_{j=0}^{s_i} E([r_{i,j}] t^{-j} \gamma_{n-i})
		\end{align*}
		is a $p$-Frobenius structure of $M_{\psi^{ext}}$. Since $E(x) \in \Z_p\llbracket x\rrbracket$, it is clear that $E_r \in 
		\mathcal{O}_L\llbracket \pi_s t^{-1}\rrbracket$.
	
	\end{proof}
	
	\subsubsection{Tame characters}
	Let $\psi:G_{\mathbb{F}_q((t))} \to \mathcal{O}_L^\times$ be a totally ramified tame character
	and let $T=(\mathbf{e}, \epsilon, \omega)$ be
	the corresponding tame ramification datum (see \S \ref{subsection: statement of main results}). Write $\epsilon=e_0+\dots + e_{a-1}p^{a-1}$ and define
	$\epsilon_j=\sum_{i=0}^{a-1} e_{i+j}p^i$.
	
	\begin{proposition} \label{proposition: Frobenius structure of Tame exponent stuff}
		The following hold:
		\begin{enumerate}
			\item The matrix $C=\diag{t^{-\epsilon_0}, \dots, t^{-\epsilon_{a-1}}}$ (resp.
			$C_0=\tcyc{t^{-e_0}, \dots, t^{-e_{a-1}}}$)
			is a dual Frobenius structure (resp. dual $p$-Frobenius structure) of $\bigoplus\limits_{j=0}^{a-1} \psi^{\otimes p^j}$ and $C=C_0^{\nu^{a-1}+\dots + \nu+ 1}$.
			\item Let $A=\diag{x_0, \dots, x_{a-1}}$ (resp. $A_0=\tcyc{y_0, \dots, y_{a-1}}$) be another dual Frobenius structure (resp. dual $p$-Frobenius structure)
			of $\bigoplus\limits_{j=0}^{a-1} \psi^{\otimes p^j}$ with $A=A_0^{\nu^{a-1}+\dots + \nu+ 1}$. 
			Then $v_t(\overline{x_j})= -\epsilon_j + n_j(q-1)$ for some
			$n_j \in \Z$ (here $\overline{x_j}$ is the image of $x_j$ in $\mathbb{F}_q((t))$). Furthermore, there exists $B=\diag{b_0, \dots, b_{a-1}}$
			with $v_t(\overline{b_j})=n_j$ such that $B^{\sigma}AB^{-1} = C$ (resp. $B^{\nu}A_0B^{-1}=C_0$).
		\end{enumerate}
	\end{proposition}
	\begin{proof}
		Let $G_{\mathbb{F}_q((t))}$ act on $\mathcal{L}=\bigoplus\limits_{j=0}^{a-1} v_j\mathcal{O}_L$ 
		via $ \bigoplus\limits_{j=0}^{a-1} \psi^{\otimes p^j}$. Let $u=t^{\frac{1}{q-1}}$ and
		let $\mathcal{E}'$ be the Amice ring over $L$ with parameter $u$. The $F$-crystal
		associated to $ \bigoplus\limits_{j=0}^{a-1} \psi^{\otimes p^j}$ is
		$(\mathcal{O}_{\mathcal{E}'} \otimes \mathcal{L})^{G_{\mathbb{F}_q((t))}}$. In particular,
		we see that $\{u^{-\epsilon_j}\otimes v_j \}$ is a basis of $(\mathcal{O}_{\mathcal{E}'} \otimes \mathcal{L})^{G_{\mathbb{F}_q((t))}}$. The first part of the proposition follows by considering the action of $\nu$ and $\sigma$ on this basis. 
		To deduce the second part of the proposition, observe what happens when
		skew-conjugating $C$ and $C_0$ by a diagonal matrix.
	\end{proof}

	\subsection{The $F$-crystal associated to $\rho$}
	We now continue with $\rho$ from \S \ref{subsection: statement of main results}
	and the setup from \S \ref{section: global bounds}.
	\subsubsection{The Frobenius structure of $\rho^{wild}$}
	Let $\mathcal{L}$ be a rank one $\mathcal{O}_L$-module on which $\pi_1^{et}(V)$ acts 
	through 
	$\rho^{wild}$. Let $f:C\to X$ be the $\Z/p^n\Z$-cover that trivializes $\rho^{wild}$. 
	Let $\overline{R}$ be the $\overline{B}$-algebra with $C \times_X V=\Spec(\overline{R})$. We may deform $\overline{B}\to \overline{R}$
	to a finite \'etale map $\widehat{B} \to \widehat{R}$. The $F$-crystal
	corresponding to $\rho$ is the $\widehat{B}$-module $M=(\widehat{R} \otimes 
	\mathcal{L})^{Gal(C/X)}$. 
	For each $Q \in W$ and $P \in f^{-1}(Q)$, we obtain a finite extension
	$\mathcal{E}_P^\dagger$ of $\mathcal{E}_Q^\dagger$ (recall from \S \ref{subsection: basic setup} that $W=\eta^{-1}(\{0,1,\infty\}$). 
	As in \S \ref{subsection: expansion around local parameters},
	we may consider the ring of overconvergent functions $R^\dagger$, which makes the following
	diagram Cartesian:
		\begin{equation*} 
	\begin{tikzcd}
	R^\dagger \arrow[d]\arrow[r] & \bigoplus\limits_{P \in f^{-1}(W)} \arrow[d] 
	\mathcal{O}_{\mathcal{E}_{P}}^\dagger \\
	\widehat{R}\arrow[r] & \bigoplus\limits_{P \in f^{-1}(W)} \mathcal{O}_{\mathcal{E}_{P}} 
	.
	\end{tikzcd}
	\end{equation*}
	Since the action of $Gal(C/X)$  (resp. $\nu$) on $\bigoplus_{P \in f^{-1}(Q)} \mathcal{O}_{\mathcal{E}_P}$ preserves $\bigoplus_{P \in f^{-1}(Q)} \mathcal{O}_{\mathcal{E}_P^\dagger}$,
	we see that $Gal(C/X)$ (resp. $\nu$) acts on $R^\dagger$ (see, e.g. \cite[\S 2]{Tsuzuki-finite_local_monodromy}). This gives the following proposition.
	
	\begin{proposition} \label{corollary: Frobenius descends to something OC}
		Let $M^\dagger = (R^\dagger \otimes 
		\mathcal{L})^{Gal(C/X)}$. The map $M^\dagger\otimes_{B^\dagger} \widehat{B}\to M$ 
		is a $\nu$-equivariant isomorphism.   
	\end{proposition}
	\begin{lemma} \label{lemma: the $F$-crystal is free}
		The module $M^\dagger$ (resp. $M$) is a free $B^\dagger$-module (resp. $\widehat{B}$-module). 
		Furthermore, $M$ has a $p$-Frobenius structure $\alpha_0$ contained in 
		$1 + \gothm B^\dagger$. 
	\end{lemma}
	\begin{proof}
		The proof of this is identical to \cite[Lemma 5.9]{kramermiller-padic}.
	\end{proof}

	\subsubsection{The Frobenius structure of  
		$\bigoplus\limits_{j=0}^{a-1} \chi^{\otimes p^{j}}$}
	\label{subsubsection: global frobenius structure tame}
	By Kummer theory, there exists $\overline{f} \in \overline{B}^\times$
	such that $\chi$ factors through the \'etale $\Z/(q-1)\Z$-cover $\Spec(\overline{B}[\overline{h}]) \to \Spec(\overline{B})$, where $\overline{h}=\sqrt[q-1]{\overline{f}}$. Let $f \in B^\dagger$ be a lift of $\overline{f}$
	and set $h=\sqrt[q-1]{f}$, so that $\Spec(B^\dagger[h]) \to \Spec(B^\dagger)$
	is an \'etale $\Z/(q-1)\Z$-cover whose special fiber is $\Spec(\overline{B}[\overline{h}]) \to \Spec(\overline{B})$. There exists $0\leq \Gamma < q-1$ such that 
	$\chi(g)=\frac{(h^\Gamma)^g}{h^\Gamma}$ for all $g \in \pi_1^{et}(V)$. Write the $p$-adic expansion $\Gamma=\gamma_0 + \dots+ \gamma_{a-1}p^{a-1}$ and define
	\begin{align*}
		\Gamma_j=\sum_{i=0}^{a-1} \gamma_{i+j} p^i.
	\end{align*}
	Note that $\chi^{\otimes p^{j}}(g)= \frac{(h^{\Gamma_j})^g}{h^{\Gamma_j}}$ for each $j$.
	This gives the following proposition. 
	\begin{proposition} \label{proposition: Global Tame Frobenius structure}
		The matrix $N=\diag{f^{-\Gamma_0}, \dots, f^{-\Gamma_{a-1}}}$ (resp.
			$N_0=\tcyc{f^{-\gamma_0}, \dots, f^{-\gamma_{a-1}}}$)
		is a dual Frobenius structure (resp. dual $p$-Frobenius structure) of $\bigoplus\limits_{j=0}^{a-1} \chi^{\otimes p^{j}}$
		and $N=N_0^{\nu^{a-1} +  \dots + 1}$. 
	\end{proposition}

	Let $Q\in W$. Recall from \S \ref{subsection: statement of main results} that
	we associate a tame ramification datum $T_Q=(\mathbf{e}_Q,\epsilon_Q,\omega_Q)$
	to $Q$ and write $\epsilon_Q=\sum e_{Q,i}p^i$. 
	The exponent of $\chi^{\otimes p^j}$ at $Q \in W$ is 
	\begin{align*}
	 \frac{\epsilon_{Q,j}}{q-1} &\mod \Z \text{, where } \\
	\epsilon_{Q,j} &= \sum_{i=0}^{a-1} e_{Q,i+j} p^i. 
	\end{align*}
	By definition we have
	\begin{align*}
	-\Div(\overline{f}^{\Gamma_j}) &= \sum_{Q \in W} (-\epsilon_{Q,j} + (q-1)n_{Q,j})[Q], 
	\end{align*}
	with $n_{Q,j} \in \Z$. Since $0\leq \epsilon_{Q,j}\leq q-2$ and $\sum_Q n_{Q,j} = \frac{\sum_Q\epsilon_{Q,j}}{q-1}$ we know
	\begin{align}
	\sum_{Q \in W} n_{Q,j} &\leq \boldm \leq  r_0 + r_\infty, \label{equation: sum of nonexponent divisor}
	\end{align}
	where we recall that $\mathbf{m}$ is the number of points where $\rho$ is ramified.
	We also have
	\begin{align}
	\begin{split}\label{equation: sum of nonexponent divisor 2}
	\sum_{j=0}^{a-1} \sum_{Q \in W} n_{Q,j} &= \frac{1}{q-1} \sum_{Q \in W} \sum_{j=0}^{a-1}\epsilon_{Q,j} \\
	&= a\Omega_\rho,
	\end{split}
	\end{align}
	where $\Omega_\rho$ is the monodromy invariant introduced in \S \ref{subsection: statement of main results}.

	\subsubsection{Comparing local and global Frobenius structures}
	\label{subsubsection: the local frob and Up}
	We fix $\gls{wild p-frob structure}$ as in Lemma \ref{lemma: the $F$-crystal is free} and
	set $\gls{wild frob structure}=\prod\limits_{i=0}^{a-1} \alpha_0^{\nu^i}$. We also
	let $\gls{tame frob structure}$ and $\gls{tame p-frob structure}$ be as in Proposition \ref{proposition: Global Tame Frobenius structure}. In particular, $\alpha N$ (resp. $\alpha_0 N_0$) is a dual Frobenius structure (resp. dual $p$-Frobenius structure) of $\rho^{wild}\otimes \bigoplus\limits_{j=0}^{a-1} \chi^{\otimes p^{j}}$. Let $Q \in W$ with $Q=P_{*,i}$. There is a map
	$\overline{B} \to \mathbb{F}_q((u_Q))$, where we expand each function on $V$
	in terms of the parameter $u_Q$. This gives a point  
	$\Spec(\mathbb{F}_q((u_{Q}))) 
	\to V$. By pulling back $\rho$ along this point we obtain
	a local representation $\rho_{Q}: 
	G_{\mathbb{F}_q((u_Q))} \to 
	\mathcal{O}_L^\times$, where $G_{\mathbb{F}_q((u_Q))}$ is the absolute Galois group of 
	$\mathbb{F}_q((u_Q))$. We will compare $\alpha_0 N_0$ to the
	local dual $p$-Frobenius structures from \S \ref{subsection: some local Frobenius structures}.

	There are three cases we need to consider.
	The first case is when $*=1$. In this case $\rho_Q^{wild}$ and $\chi_Q$ are both unramified. This is because $\rho$ is only
	ramified at the points $\tau_1,\dots,\tau_{\mathbf{m}}$ and by Lemma \ref{lemma: map to p1} 
	we have $\eta(\tau_i)\in \{0,\infty\}$. The second case is when $*\in \{0,\infty\}$
	and $\rho_Q^{wild}$ is unramified. The last case is when $* \in \{0,\infty\}$ and
	$\rho_Q^{wild}$ is ramified. In each case, we will describe
	a dual $p$-Frobenius structure $\gls{local nice frobenius structure}$ of $\rho_Q^{wild} \otimes \bigoplus\limits_{j=0}^{a-1} \chi_Q^{\otimes p^{j}}$, an element $\gls{wild local change of basis} \in \mathcal{O}_{\mathcal{E}_Q}^\dagger$,
	and a diagonal matrix $\gls{tame local change of basis} \in GL_{a}(\mathcal{O}_{\mathcal{E}^\dagger})$ satisfying:
	\begin{align} \label{equation: a-th root of frobenius}
	\begin{split}
	(b_QM_Q)^{\nu}\alpha_0 N_0 (b_QM_Q)^{-1} &= C_Q, \\
	(b_QM_Q)^{\sigma}\alpha N(b_QM_Q)^{-1}&=C_Q^{\nu^{a-1}+\nu^{a-2}+ \dots + 1}.
	\end{split}
	\end{align}
	The dual $p$-Frobenius structure $C_Q$ will be closely related to the
	dual $p$-Frobenius structures studied in \S \ref{subsection: some local Frobenius structures}.
	It is helpful for us to introduce the following rings: 
	\begin{align*}
	\mathcal{R}_Q &= \bigoplus_{j=0}^{a-1} \mathcal{E}_Q, ~~~~~\mathcal{O}_{\mathcal{R}_Q} = \bigoplus_{j=0}^{a-1} \mathcal{O}_{\mathcal{E}_Q}, \\
	\gls{bounded robba in uQ} &= \bigoplus_{j=0}^{a-1} \mathcal{E}_Q^\dagger, ~~~~~\gls{integral bounded robba in uQ} = \mathcal{R}_Q^\dagger \cap \mathcal{O}_{\mathcal{R}_Q}.
	\end{align*}
	We define $u_{Q,j} \in \mathcal{R}_Q$ to have $u_Q$ in the $j$-th coordinate and zero in the other coordinates.
	For each $Q$ we will define a subspace $\gls{special convergence space}\subset \mathcal{R}_Q^\dagger$ of elements satisfying some precise
	convergence conditions.
	
	\begin{enumerate}[label=\Roman*.]
		\item If $*=1$, then $\nu_Q$ sends $u_Q \mapsto \sqrt[p-1]{(u_Q^{p-1}+1)^p-1}$ (see
		the end of \S \ref{subsection Global Frobenius and Up}). 
		\begin{enumerate}
			\item[(wild)] As $\rho_Q$ is unramified, we know from Proposition \ref{proposition: unramified character proposition} 
			there exists $b_Q \in 1 + \gothm\mathcal{O}_{\mathcal{E}_Q}^\dagger$ such
			that the dual $p$-Frobenius structure $c_Q=\frac{b_Q^\nu}{b_Q}\alpha_0$ of $\rho_Q^{wild}$ lies in $1+\gothm$. 
			\item[(tame)] Since $\chi_Q$ is unramified, the exponent is zero. By Proposition \ref{proposition: Frobenius structure of Tame exponent stuff} there exists $M_Q=\diag{m_{Q,0}, \dots, m_{Q,a-1}}$
			with $v_{u_Q}(\overline{m_{Q,j}})=n_{Q,j}$ such that 
			$M_Q^{\sigma} N M_{Q}^{-1}=\diag{1, \dots, 1}$ and $M_Q^{\nu} N_0 M_{Q}^{-1}=\tcyc{1,\dots, 1}$. 
			
			\item[(both)] We see that $C_Q=\tcyc{c_Q,\dots,c_Q}$ is a 
			dual $p$-Frobenius structure of $\rho_Q^{wild} \otimes \bigoplus\limits_{j=0}^{a-1} \chi_Q^{\otimes p^{j}}$ and that \eqref{equation: a-th root of frobenius} holds.
			Define $\mathcal{O}_{\mathcal{R}_Q}^{con}$ to 
			be $\bigoplus\limits_{j=0}^{a-1} \mathcal{D}$ viewed as a subspace
			of $\mathcal{O}_{\mathcal{R}_Q^\dagger}$ (see \eqref{definition of A} for the
			definition of $\mathcal{D}$).
			From Proposition \ref{proposition: type 2 Frobenius} we have
			\begin{align} 
			\begin{split}\label{UP computation type 1}
			U_p \circ C_Q (p^{b(k+pn)}u_{Q,j}^{-(k+pn)}) & 
			\in 
			p^{n}\mathcal{O}_{\mathcal{R}_Q}^{con}, \\
			U_p \circ C_Q (\mathcal{O}_{\mathcal{R}_Q}^{con}) &\subset \mathcal{O}_{\mathcal{R}_Q}^{con}.
			\end{split}
			\end{align}

		\end{enumerate}
		
		\label{Type A Frobenius blurb}
		
		\item Next, consider the case where $*$ is $0$ or $\infty$ and $\rho^{wild}_{Q}$
		is unramified. 
		Then $\nu_Q$ sends $u_Q \mapsto u_Q^p$. We choose $\mathfrak{s}_Q \in \Q$
		such that the following hold:
		\begin{align}\label{equation: false wild ramification number}
		\begin{split}
			\pi_{\mathfrak{s}_Q} &\in \mathcal{O}_E, \\
			\frac{1}{\mathfrak{s}_Q} - \frac{\omega_Q}{a\mathfrak{s}_Q(p-1)}&\geq 1. 
		\end{split}
		\end{align}
		\begin{enumerate}
			\item[(wild)] From Proposition \ref{proposition: unramified character proposition} there exists $b_Q\in 
			1+\gothm \mathcal{O}_{\mathcal{E}_Q}^\dagger$ such that 
			$c_Q=\frac{b_Q^\nu}{b_Q} \alpha_0  \in 1+\gothm$ is a dual $p$-Frobenius structure
			of $\rho_Q^{wild}$.
			
			\item[(tame)] By Proposition \ref{proposition: Frobenius structure of Tame exponent stuff} there exists $M_Q=\diag{m_{Q,0}, \dots, m_{Q,a-1}}$
			with $v_{u_Q}(\overline{m_{Q,j}})=n_{Q,j}$ such that $M_Q^{\sigma} N M_{Q}^{-1}=\diag{u_Q^{-\epsilon_{Q,0}}, \dots, u_Q^{-\epsilon_{Q,a-1}}}$ and
				$M_Q^{\nu} N_0 M_{Q}^{-1}=\tcyc{u_Q^{-e_{Q,0}},\dots, u_Q^{-e_{Q,a-1}}}$. 
					
			\item[(both)] We see that $C_Q=\tcyc{c_Qu_Q^{-e_{Q,0}},\dots,c_Qu_Q^{-e_{Q,a-1}}}$ is a 
						dual $p$-Frobenius structure of $\rho_Q^{wild}\otimes\bigoplus\limits_{j=0}^{a-1} \chi_Q^{\otimes p^{j}}$ and that \eqref{equation: a-th root of frobenius} holds.
						Define $\mathcal{O}_{\mathcal{R}_Q}^{con}$ to 
						be a copy of $\mathcal{D}_{\mathbf{e}_Q,\mathfrak{s}_Q}$ in
						$\mathcal{O}_{\mathcal{R}_Q^\dagger}$ (recall the definition of $\mathcal{D}_{\mathbf{e},s}$ from \S \ref{subsubsection: type one frob local estimates}). From Proposition \ref{proposition: type 1 Up estimates}
						we have 
						\begin{align}
						\begin{split}\label{UP computation type 2}
						U_p \circ C_Q (\pi_{a\mathfrak{s}_Q}^{q(\mathbf{e}_Q,j)} \pi_{\mathfrak{s}_Q}^{pn} u_{Q,j}^{-n}) &\in\pi_{\mathfrak{s}_Q}^{n(p-1)} \pi_{a\mathfrak{s}_Q}^{-\omega_Q} \mathcal{O}_{\mathcal{R}_Q}^{con}, \\
						U_p \circ C_Q (\mathcal{O}_{\mathcal{R}_Q}^{con}) &\subset \mathcal{O}_{\mathcal{R}_Q}^{con}.
						\end{split}
						\end{align}
			\end{enumerate}

		\item Finally, we consider the case where $*$ is $0$ or $\infty$ and $\rho^{wild}_{Q}$
		is ramified. 
		Then $\nu_Q$ sends $u_Q \mapsto u_Q^p$. 
		\begin{enumerate}
			\item[(wild)] By Proposition \ref{theorem: 
				ASW frobenius structure} there is $b_Q\in 
			1+\gothm \mathcal{O}_{\mathcal{E}_Q}^\dagger$ such that 
			$c_Q=\frac{b_Q^\nu}{b_Q} \alpha_0  \in \mathcal{O}_L\llbracket  \pi_{s_Q} u_Q^{-1}\rrbracket$
			is a dual $p$-Frobenius structure of $\rho_Q^{wild}$ (recall $s_Q$ is
			the Swan conductor of $\rho$ at $Q$). Note that $c_Q \equiv 1 \mod \gothm$.
			
			\item[(tame)] By Proposition \ref{proposition: Frobenius structure of Tame exponent stuff} there exists $M_Q=\diag{m_{Q,0}, \dots, m_{Q,a-1}}$
			with $v_{u_Q}(\overline{m_{Q,j}})=n_{Q,j}$ such that $M_Q^{\sigma} N M_{Q}^{-1}=\diag{u_Q^{-\epsilon_{Q,0}}, \dots, u_Q^{-\epsilon_{Q,a-1}}}$ and
			$M_Q^{\nu} N_0 M_{Q}^{-1}=\tcyc{u_Q^{-e_{Q,0}},\dots, u_Q^{-e_{Q,a-1}}}$.

			\item[(both together)] We see that $C_Q=\tcyc{c_Qu_Q^{-e_{Q,0}},\dots,c_Qu_Q^{-e_{Q,a-1}}}$ is a 
			dual $p$-Frobenius structure of $\rho_Q^{wild}\otimes\bigoplus\limits_{j=0}^{a-1} \chi_Q^{\otimes p^{j}}$ and that \eqref{equation: a-th root of frobenius} holds.
			We
			define $\mathcal{O}_{\mathcal{R}_Q}^{con}$ to 
			be a copy of $\mathcal{D}_{\mathbf{e}_Q,s_Q}$ in
			$\mathcal{O}_{\mathcal{R}_Q^\dagger}$. From Proposition \ref{proposition: type 1 Up estimates}
			we see that 
			\begin{align}
			\begin{split}\label{UP computation type 3}
			U_p \circ C_Q (\pi_{a s_Q}^{q(\mathbf{e}_Q,j)}\pi_{s_Q}^{pn} u_{Q,j}^{-n}) &\in 
			\pi_{s_Q}^{n(p-1)}\pi_{as}^{-\omega_Q}\mathcal{O}_{\mathcal{R}_Q}^{con}, \\
			U_p \circ C_Q (\mathcal{O}_{\mathcal{R}_Q}^{con}) &\subset \mathcal{O}_{\mathcal{R}_Q}^{con}.
			\end{split}
			\end{align}
		\end{enumerate}

	\end{enumerate}
	
	\subsubsection{Comparing global and semi-local Frobenius structures} \label{paragraph: global to semi-local setup}
	We define the following spaces:
	\begin{align*}
	\begin{split} 
	\gls{R} &= \bigoplus_{Q \in W} 
	\mathcal{R}_Q,  ~~~~~~
	\gls{Rdagger}=\bigoplus_{Q \in W} 
	\mathcal{R}_Q^\dagger,\\
	\gls{RtrunQ}&=\begin{cases}
		\bigoplus\limits_{j=0}^{a-1} \mathcal{E}_{Q}^{\leq -1} & \eta(Q)=0,\infty \\
		\bigoplus\limits_{j=0}^{a-1} \mathcal{E}_{Q}^{\leq -p} & \eta(Q)=1,
	\end{cases} \\
	\gls{Rtrun}&= \bigoplus_{Q \in W} 
	\mathcal{R}^{trun}_Q,~~~~~~\gls{OconR}= \bigoplus_{Q\in W} \mathcal{O}_{\mathcal{R}_Q}^{con} \subset \mathcal{R}^\dagger. 
	\end{split}
	\end{align*}
	Define $\mathcal{O}_{\mathcal{R}}$ to be $\bigoplus\limits_{Q \in W} 
	\mathcal{O}_{\mathcal{R}_Q}$ and define $\mathcal{O}_{\mathcal{R}^{trun}}$ to be $\mathcal{R}^{trun} \cap \mathcal{O}_{\mathcal{R}}$. Note that $\mathcal{O}_{\mathcal{R}}^{con}$ is contained in 
	$\mathcal{O}_{\mathcal{R}}$.
	There is a projection map $pr:\mathcal{R} \to \mathcal{R}^{trun}$, which
	is the direct sum of the projection maps described in \S \ref{subsection: basic definitions}.
	By the definition of each summand of $\mathcal{O}_{\mathcal{R}}^{con}$ we see that 
	\begin{align}\label{equation: projecting onto tails ends up in W}
		\ker(pr)\cap \mathcal{O}_\mathcal{R} &\subset \mathcal{O}_{\mathcal{R}}^{con}.
	\end{align}
	We
	may view $\bigoplus\limits_{j=0}^{a-1} \widehat{\mathcal{B}}$ (resp. $\bigoplus\limits_{j=0}^{a-1} \mathcal{B}^\dagger$) as a subspace of $\mathcal{R}$ (resp. $\mathcal{R}^\dagger)$ using the maps in \eqref{overconvergent expansion diagram}. Let $\gls{C}$ (resp. $\gls{M}$ and
	$\gls{b}$) denote the endomorphism of $\mathcal{R}^\dagger$ 
	that acts on the $Q$-coordinate by
	$C_Q$ (resp. $M_Q$ and $\diag{b_Q,\dots,b_Q}$). This gives an operator $U_p \circ C: \mathcal{R}^\dagger \to \mathcal{R}^\dagger$. From \eqref{UP computation type 1}, \eqref{UP computation type 2},
	and \eqref{UP computation type 3} we have
	\begin{align} \label{equation: W is preserved}
		U_p \circ C (\mathcal{O}_{\mathcal{R}}^{con}) \subset \mathcal{O}_{\mathcal{R}}^{con}.
	\end{align}
	Also, by \eqref{equation: a-th root of frobenius} know
	\begin{align}\label{local change of Frobenius eq with a-root}
	\begin{split}
		(bM)^{\nu}\alpha_0 N_0(bM)^{-1}&= C, \\
		(bM)^{\sigma}\alpha N(bM)^{-1}&=C^{\nu^{a-1}+\nu^{a-2}+ \dots + 1}.
	\end{split}
	\end{align} 
	For each $Q$ we have
	\begin{align} \label{local change of frob is eq 1 mod p}
		\begin{split}
		b_Q  &\equiv 1 \mod \gothm, \\
		M_Q &\equiv \diag{u_{Q,0}^{n_{Q,0}}g_0, \dots, u_{Q,a-1}^{n_{Q,a-1}} g_{a-1}} \mod \gothm, 
		\end{split}
	\end{align}
	with $g_j \in \mathbb{F}_q\llbracket u_{Q,j}\rrbracket^\times$.

	\section{Normed vector spaces and Newton polygons}
	\label{section: newton polygons and functional analysis}
	For the convenience of the reader, we recall some definitions and facts
	about Newton polygons and normed $p$-adic vector spaces. 
	Most of what follows is well known (see e.g.
	\cite{Serre-p-adic_banach} or \cite{Monsky-forma_cohomology3}
	for many standard facts on $p$-adic functional analysis). However, we do
	find it necessary to introduce some notation and definitions that are not standard.
	In particular, we introduce
	the notion of a \emph{formal basis}, which allows us to compute Fredholm determinants by estimating columns (in contrast to estimating rows, which is the approach taken in \cite{Adolphson-Sperber-exponential_sums}). 
	
	\subsection{Normed vector spaces and Banach spaces}
	\label{subsection: normed vector spaces and Banach spaces}
	Let $V$ be a vector space over $L$ with a norm $|\cdot|$
	compatible with the $p$-adic norm $|\cdot |_p$ on $L$. We
	will assume that for every $x \in V$ the norm $|x|$
	lies in $|L|_p$, the norm group of $L$. We say $V$
	is a \emph{Banach space} if it is also complete. Let $V_0\subset V$ 
	denote the subset consisting of $x \in V$ satisfying $|x|\leq 
	1$
	and let $\overline{V}=V_0/\gothm V_0$. If $W$ is a subspace 
	of $V$, we
	automatically give $W$ the subspace norm unless otherwise specified. 
	\begin{definition}
		Let $I$ be a set. We let $\mathbf{s}(I)$ denote the set of families
		$x=(x_i)_{i\in I}$ with $x_i \in L$ such that $|x|=\sup\limits_{i\in I}|x_i|_p < \infty$. 
		Then $\mathbf{s}(I)$ is a Banach space with the norm $|\cdot|$. We let
		$\mathbf{c}(I) \subset \mathbf{s}(I)$ be the subspace of families
		with $\lim\limits_{i\in I} x_i = 0$ (note that this agrees with
		$\mathbf{c}(I)$ defined in \cite[\S I]{Serre-p-adic_banach}).  
	\end{definition}
	
	\begin{definition}
		A \emph{formal basis} of $V$ is a subset $G=\{e_{i}\}_{i \in I} \subset V$
		with a norm preserving embedding $V \to \mathbf{s}(I)$
		where $e_i$ gets mapped to the element in $\mathbf{s}(I)$ with
		$1$ in the $i$-coordinate and $0$ otherwise.\footnote{In \cite{kramermiller-padic} we use the term \emph{integral basis}.} We regard $V$ a subspace of $\mathbf{s}(I)$.
	\end{definition}
	\begin{definition}
		An \emph{orthonormal basis} of $V$ is a formal basis $G=\{e_{i}\}_{i \in I} \subset V$ such that 
		$V \subset \mathbf{c}(I)$. This inclusion is an equality if $V$ is a Banach space. By \cite[Proposition I]{Serre-p-adic_banach},
		every Banach space over $L$ has an orthonormal basis. Thus, every Banach space 
		is of the form $\mathbf{c}(I)$. 
	\end{definition}
	
	\begin{example} \label{example: banach space stuff}
		Let $V$ be the Banach space $\mathcal{O}_L\llbracket t\rrbracket\otimes \Q_p$. Then $\{t^n\}_{n \in \Z_{\geq 0}}$
		is a formal basis of $V$ and there is an isomorphism $V \cong \mathbf{s}(\Z_{\geq 0})$.
		By \cite[Lemme I]{Serre-p-adic_banach} any orthonormal basis
		of $V$ reduces to an $\mathbb{F}_q$-basis of $\overline{V}=\mathbb{F}_q\llbracket t\rrbracket$
		and thus must be uncountable.
		The Tate algebra $L\ang{t} \subset V$ is a Banach space,
		which we may identify with $\mathbf{c}(\Z_{\geq 0})$. 
	\end{example}
	
	\subsubsection{Restriction of scalars to $E$}
	Let $I$ be a set. Assume that $V \subset \mathbf{s}(I)$ has $G$
	as a formal basis. We may regard $V$ as a vector space over $E$. Let
	$\zeta_1=1, \zeta_2,\dots,\zeta_a \in \mathcal{O}_L$ be elements
	that reduce to a basis of $\mathbb{F}_q$ over $\mathbb{F}_p$ modulo $\pi_\circ$ and set
	$I_E=I \times \{1,\dots, a\}$. We define 
	\begin{align*}
	G_E &= \{ \zeta_j e_i\}_{(i,j)\in I_E}.
	\end{align*}
	Note that $G_E$ is a formal basis of $V$ over $E$. 

	\subsection{Completely continuous operators and Fredholm determinants}
	\label{subsection: nuclear operators, etc.}
	
	\subsubsection{Completely continuous operators}
	\label{subsubsection: Completely continuous operators}
	Let $V$ be a vector space over $L$ with norm $|\cdot|$. Let $G=\{e_{i}\}_{i \in I}$
	be a formal basis of $V$. We assume $I$ is a countable set.
	Let $u:V \to V$ (resp. $v: V \to V$) be an
	$L$-linear (resp. $E$-linear) operator. Let $(n_{i,j})$
	be the matrix of $u$ with respect to the basis
	$G$. 

	\begin{definition}
		For $i \in I$, we define $\textbf{row}_i(u,G)=\inf\limits_{j \in I} v_p(n_{i,j})$
		and $\textbf{col}_i(u,G)=\inf\limits_{j \in I} v_p(n_{j,i})$. That is,
		$\mathbf{row}_i(u,G)$ (resp. $\mathbf{col}_i(u,G)$) is the smallest $p$-adic
		valuation that occurs in the $i$-th row (resp. column) of the matrix of $u$.
		Note that $\textbf{col}_i(u,G)=\log_p|u(e_i)|$.
	\end{definition}

	\begin{definition}  \label{definition: completely continuous}
		Assume that $V=\mathbf{c}(I)$. We say that $u$ is \emph{completely continuous} if it
		is the $p$-adic limit of $L$-linear operators 
		with finite dimensional image. This is equivalent to $\lim\limits_{i \in I} \mathbf{row}_i(u,G) = \infty$
		(see \cite[Theorem 6.2]{Monsky-padic_notes}). We make
		the analogous definition for $v$.
	\end{definition}

	\subsubsection{Fredholm determinants}\label{subsubsection: fredholm determinants}
	We continue with the notation from \S \ref{subsubsection: Completely continuous operators}.
	We define the \emph{Fredholm determinant} of $u$ with respect to $G$ to be the formal sum
	\begin{align} \label{Fredholm definition}
	\begin{split}
	\det(1-su| G) &= \sum_{n=0}^\infty c_ns^n, \\
	c_n &= (-1)^n \sum_{\stackrel{S \subset I}{|S|=n}} 
	\sum_{\sigma \in \Sym(S)} \text{sgn}(\sigma) \prod_{i \in S} n_{i,\sigma(i)}.
	\end{split}
	\end{align}
	We define the Fredholm determinant $\det\limits_E(1-sv|G_E)$ in
	an analogous manner using the matrix of $v$ with respect to $G_E$.
	Note that there is no reason a priori for the sums $c_n$ to converge. 
	We will say that $\det(1-su|G)$ is \emph{well-defined} if each $c_n$ converges.
	
	\begin{lemma} \label{lemma: completely continuous does not depend on basis}
		Assume $V$ is a Banach space with orthonormal basis $G$ and that $u$ is completely continuous. 
		Then $\det(1-su|G)$ is well-defined and is an entire function in $s$. Furthermore, if $G'$
		is another orthonormal basis of $V$ we have $\det(1-su|G)=\det(1-su|G')$.
		The analogous result holds for $v$.
	\end{lemma}
	\begin{proof}
		See \cite[Proposition 7]{Serre-p-adic_banach}.
	\end{proof}
	\begin{definition}
		Continue with the notation from Lemma \ref{lemma: completely continuous does not depend on basis}. 
		We let $\det(1-su|V)$ denote the Fredholm determinant $\det(1-su|G)$. By Lemma \ref{lemma: completely continuous does not depend on basis} this determinant does not depend on our choice of orthonormal basis.
		We define $\det\limits_E (1-sv|V)$ similarly.
	\end{definition}
	
	\subsubsection{Newton polygons of operators}
	\begin{definition}
		Let $*$ be either $p$ or $q$. Let $f(t)=\sum a_nt^n \in L\langle t\rangle^\times$
		be an entire function. We
		define the $*$-adic Newton polygon $NP_*(f)$ to be the lower convex hull of the points $(n,v_*(a_n))$. For $r>0$, we
		let $NP_*(f)_{<r}$ denote the ``sub-polygon'' of $NP_*(f)$ consisting of all segments whose
		slope is less than $r$. 
	\end{definition}
	\begin{definition}
		Adopt the notation from \S \ref{subsubsection: fredholm determinants}. Assume that
		$\det(1-su|G)$ (resp. $\det\limits_E(1-sv|G_E)$) is well-defined and an entire function in $s$.
		Then we define $NP_*(u|G)$ (resp. $NP_*(v|G_E)$) to be $NP_*(\det(1-su|G))$ (resp. $NP_*(\det\limits_E(1-sv|G_E))$).
		Further assume that $V$ is a Banach space and $u$ (resp. $v$) is completely continuous. Then by
		Lemma \ref{lemma: completely continuous does not depend on basis}, the Fredholm determinant does not depend
		on the choice of orthonormal basis, so we define $NP_*(u|V)$ (resp. $NP_*(v|V)$) to be $NP_*(u|G)$ (resp.
		$NP_*(v|G_E)$). 
	\end{definition}
	\begin{definition}
		Let $d \in \Z_{\geq 1} \cup \infty$ and let $A=\{c_n\}^{d}_{n\geq 1}$ be a non-decreasing sequence of real numbers. If $d=\infty$ we will make the assumption that $\lim\limits_{n \to \infty} c_n = \infty$. Let $P_A$ be the ``polygon''
		of length $d$ consisting of vertices $(0,0),(1,c_1),(2,c_1+c_2), \dots$. We write
		\begin{align*}
			NP_*(f) &\succeq A
		\end{align*}
		if the polygon $NP_*(f)$ lies above $P_A$ at every $x$-coordinate where both are defined. 
	\end{definition}

	\noindent The following lemma allows us to bound $NP_p(v|G_E)$ by estimating the columns of the matrix representing $v$.
	\begin{lemma} \label{lemma: estimating NP by estimating columns}
		Assume that $\lim\limits_{i \in I} \mathbf{col}_{(i,1)}(v,G_E) = \infty$.
		If $v$ is $\nu^{-1}$-semilinear, then the Fredholm determinant $\det(1-sv|G_E)$ is well-defined and we have
		\begin{align*}
		NP_p(v|G_E) &\succeq  \{\mathbf{col}_{(i,1)}(v,G_E)\}^{\times a}_{i \in I},
		\end{align*}
		where the superscript $\times a$ means each slope is repeated $a$ times.
	\end{lemma}
	\begin{proof}
		Note that $v(\zeta_j e_i)=\zeta_j^{\nu^{-1}}v(e_i)$, which implies
		 $\mathbf{col}_{(i,j)}(v,G_E)=\mathbf{col}_{(i,1)}(v,G_E)$ for each $j$.
		 In particular, $\lim\limits_{(i,j) \in I_E} \mathbf{col}_{(i,j)}(v,G_E) = \infty$, so $\det\limits_E(1-sv|G_E)$ is well-defined. By the definition of $c_n$ in \eqref{Fredholm definition} we see that 
		 \begin{align*}
			 NP_p(v|G_E) &\succeq \{\mathbf{col}_{(i,j)}(v,G_E)\}_{(i,j) \in I_E} \\
			 &=\{\mathbf{col}_{(i,1)}(v,G_E)\}^{\times a}_{i \in I}.
		 \end{align*}
	\end{proof}

\subsubsection{Computing Newton polygons using $a$-th roots}
When estimating the Newton polygon of an $L$-linear completely continuous operator $u$ on $V$, it is convenient
to work with an $E$-linear operator $v$ that is an $a$-th root of $u$. The reason
is that we can translate $p$-adic bounds on $\det\limits_E(1-sv|V)$ to $q$-adic bounds on $\det(1-su|V)$.

\begin{lemma} \label{lemma: a-th root lemma}
	Let $V$ be a Banach space. 
	Let $v$ be a completely continuous $E$-linear operator on $V$ and let $u=v^a$. 
	Assume that $u$ is $L$-linear.
	We further assume that $\det(1-su|V)$ has coefficients in $E$ (a priori its
	coefficients could lie in $L$). Let $\frac{1}{a}NP_p(v|V)$ denote
	the polygon where both the $x$-coordinates and $y$-coordinates of the points
	in $NP_p(v|V)$ are scaled by a factor of $\frac{1}{a}$. 
	Then
	$NP_q(u|V)=\frac{1}{a}NP_p(v|V)$.
\end{lemma}

\begin{proof}
	Some version of this lemma is present in most papers proving ``Hodge bounds'' for
	exponential sums (see e.g. \cite{Bombieri-exponential_sums} or \cite{Adolphson-Sperber-exponential_sums}). The proof of
	\cite[Lemma 6.25]{kramermiller-padic} is easily adapted to our situation. 
\end{proof}

	\section{Finishing the proof of Theorem \ref{main theorem}}

	\label{section: proof of theorems}

	\subsection{The Monsky trace formula}
	\label{subsection: MW trace formula}
	Let us recall the Monsky trace formula
	in the case of curves. For a complete treatment see \cite{Monsky-forma_cohomology3} 
	or \cite[\S 10]{Wan-higher_rank_dwork_conjecture}.  
	Let
	$\Omega_{\mathcal{B}^\dagger}^i$ denote the space of
	$i$-forms of $\mathcal{B}^\dagger$ (see \cite[\S 4]{Monsky-forma_cohomology3}).
	The map $\sigma$ induces a map $\sigma_i: \Omega^i_{\mathcal{B}^\dagger} \to 
	\Omega^i_{\mathcal{B}^\dagger}$
	sending $xdy$ to $x^\sigma d(y^\sigma)$. As in
	\cite[\S 3]{van_der_Put-MW_cohomology}, there exist trace maps
	$
	\text{Tr}_i: \Omega_{\mathcal{B}^\dagger}^i \to 
	\sigma(\Omega_{\mathcal{B}^\dagger}^i)$.
	Let $\Theta_i$ denote the map $\sigma_i^{-1} \circ \text{Tr}_i$.
	For $\omega \in \Omega^1_{\mathcal{B}^\dagger}$ and $x \in \mathcal{B}^\dagger$ we have
	\begin{align} \label{Dwork operator property on forms}
	\Theta_1(x\omega^\sigma) &= \Theta_0(x)\omega.
	\end{align}	
	\noindent Consider the $L$-function
	\begin{align} \label{introduction of L-function:2}
	L(\rho^{wild}\otimes \chi^{\otimes p^j},V,s)&= \prod_{x \in V} \frac{1}{1 - \rho^{wild}\otimes \chi^{\otimes p^j}(Frob_x) s^{\deg(x)}},
	\end{align}
	which is a slight modification of \eqref{introduction of L-function}. 
	Fix a tuple $\mathbf{r}=(r_Q)_{Q \in W}$ of positive rational numbers.
	Monsky shows that if the $r_Q$'s are sufficiently small (so $\mathcal{B}(0,\mathbf{r}]$
	consists of functions with sufficiently small radius of overconvergence), the operator
	$\Theta_i \circ \alpha f^{-\Gamma_j}$ is completely continuous on $\Omega_{\mathcal{B}(0,\mathbf{r}]}^i$.
	The Monsky trace formula states
	\begin{align} \label{Main Monsky-Washnitzer}
	L(\rho^{wild}\otimes \chi^{\otimes p^j},V,s) &= \frac{\det(1-s\Theta_1 \circ \alpha f^{-\Gamma_j}| \Omega_{\mathcal{B}(0,\mathbf{r}]}^1)}
	{\det(1-s\Theta_0 \circ \alpha f^{-\Gamma_j}| \mathcal{B}(0,\mathbf{r}])},
	\end{align}
	where $\alpha$, $f$, and $\Gamma_j$ are as in \S \ref{subsubsection: global frobenius structure tame}-\ref{subsubsection: the local frob and Up}.
	Thus, we may estimate $L(\rho^{wild}\otimes \chi^{\otimes p^j},V,s)$ by
	estimating operators on the space of $1$-forms and $0$-forms.
	
	In our situation we may simplify \eqref{Main Monsky-Washnitzer}. 
	The map $\mathcal{A}^\dagger \to \mathcal{B}^\dagger$ is \'etale,
	which implies $\Omega_{\mathcal{B}^\dagger} = \pi^* \Omega_{\mathcal{A}^\dagger}$. 
	Since
	$\Omega_{\mathcal{A}^\dagger}=\mathcal{A}^\dagger \frac{dt}{t}$, we see that 
	$\Omega_{\mathcal{B}^\dagger} = \mathcal{B}^\dagger \frac{dt}{t}$. In particular, we have
	$\Omega_{\mathcal{B}(0,\mathbf{r}]} = \mathcal{B}(0,\mathbf{r}] \frac{dt}{t}$.
	Also, since $\frac{dt}{t}= \frac{1}{q}(\frac{dt}{t})^\sigma$ we know by
	\eqref{Dwork operator property on forms} that
	$\Theta_1\Big (x\frac{dt}{t}\Big ) =\frac{1}{q} \Theta_0(x) \frac{dt}{t}$. Thus, we have $\Theta_1=U_q$ and $\Theta_0=qU_q$. Then \eqref{Main 
		Monsky-Washnitzer}
	becomes
	\begin{align} \label{MW via characteristic}
	L(\rho^{wild}\otimes \chi^{\otimes p^j},V,s) &= \frac{\det(1-sU_q \circ \alpha f^{-\Gamma_j}| \mathcal{B}(0,\mathbf{r}])}
	{\det(1-sqU_q \circ \alpha f^{-\Gamma_j}| \mathcal{B}(0,\mathbf{r}])}.
	\end{align}
	As $\det(1-sU_q \circ \alpha f^{-\Gamma_j}| \mathcal{B}(0,\mathbf{r}]) \in 
	1+s\mathcal{O}_L\llbracket s\rrbracket$, we know $\frac{1}{\det(1-sqU_q \circ \alpha f^{-\Gamma_j}| \mathcal{B}(0,\mathbf{r}])}$ lies in $ 1+qs\mathcal{O}_L\llbracket qs\rrbracket$.
	This means each slope of $NP_q\Big(\frac{1}{\det(1-sqU_q \circ \alpha f^{-\Gamma_j}| \mathcal{B}(0,\mathbf{r}])}\Big)$ is at least one. In particular, we have
	\begin{align*} 
	NP_q(L(\rho^{wild}\otimes \chi^{\otimes p^j},V,s))_{<1} &= NP_q(U_q \circ \alpha f^{-\Gamma_j}| \mathcal{B}(0,\mathbf{r}])_{<1}.
	\end{align*}
	Note that $\rho$ and $\rho^{wild}\otimes \chi^{\otimes p^j}$ are Galois conjugates. Thus, $L(\rho,V,s)$ and $L(\rho^{wild}\otimes \chi^{\otimes p^j},V,s)$ are Galois conjugates.
	This gives
	\begin{align} \label{equation: L function estimate comes from U_p}
	\begin{split}
		NP_q(L(\rho,V,s))_{<1} &= \frac{1}{a} NP_q(L(\rho^{wild}\otimes\bigoplus_{j=0}^{a-1} \chi^{\otimes p^j},V,s))_{<1} \\
		&= \frac{1}{a} NP_q(U_q \circ \alpha N| \bigoplus_{j=0}^{a-1}\mathcal{B}(0,\mathbf{r}])_{<1},
	\end{split}
	\end{align}
	where $N$ is the dual Frobenius structure from Proposition \ref{proposition: Global Tame Frobenius structure}.
	
	\subsection{Estimating $NP_q(U_q \circ \alpha N| \bigoplus\limits_{j=0}^{a-1}\mathcal{B}(0,\mathbf{r}])$}
	\label{subsection: estimating NP_p}
	In this subsection we estimate the $q$-adic Newton polygon of $U_q\circ \alpha N $
	acting on $ \bigoplus\limits_{j=0}^{a-1} \mathcal{B}(0,\mathbf{r}]$.

	\begin{proposition}  \label{proposition: NP bounds with operator assumption1}
		We have
		\begin{align*}
		\frac{1}{a}NP_q(U_q \circ \alpha N| \bigoplus\limits_{j=0}^{a-1}\mathcal{B}(0,\mathbf{r}])_{<1} &\succeq \big 
		\{\underbrace{0,\dots,0}_{g-1+r_0+r_1+r_\infty-\Omega_\rho}
		\big \}
		\bigsqcup \Bigg ( \bigsqcup_{i=1}^{\mathbf{m}} S_{\tau_i} \Bigg ),
		\end{align*}
		where $S_{\tau_i}$ is the slope-set defined in \S \ref{subsection: statement of main results}
		and $r_*$ is the cardinality of $\eta^{-1}(*)$ defined in \S \ref{subsection: basic setup}.
	\end{proposition}
	\noindent We break the proof up into several steps. 
	
	\paragraph{The twisted space and the $a$-th root.}
	We view $ \bigoplus\limits_{j=0}^{a-1} \widehat{\mathcal{B}}$ and $ \bigoplus\limits_{j=0}^{a-1} \mathcal{B}(,\mathbf{r}]$ as subspaces of $\mathcal{R}$,
	as described in \S \ref{paragraph: global to semi-local setup}.
	Unfortunately, the global Frobenius structure $\alpha N$ will not have nice growth
	properties like the local Frobenius structures studied in \S \ref{subsection: local rank one crystals}. Instead, we have to ``twist'' this subspace using the matrices
	$bM$ defined in \S \ref{paragraph: global to semi-local setup}. 
	Define the spaces
	\begin{align*}
		\widehat{V}&=bM\Big(  \bigoplus\limits_{j=0}^{a-1} \widehat{\mathcal{B}}\Big),\\
		V&=bM\Big(  \bigoplus\limits_{j=0}^{a-1} \mathcal{B}(0,\mathbf{r}]\Big),
	\end{align*}
	which we regard as subspaces of $\mathcal{R}$.
	After decreasing $\mathbf{r}$ we may assume that
	\begin{align}\label{eq: growth is the same before and after the twist}
		V &= \widehat{V} \bigcap \bigoplus_{Q \in W} \bigoplus_{j=0}^{a-1} \mathcal{E}_Q(0,r_Q]. 
	\end{align}
	In fact, \eqref{eq: growth is the same before and after the twist} holds as long as $bM$,
	viewed as a matrix with elements in $\bigoplus\limits_{Q \in W} \mathcal{E}_Q^{\dagger}$, has entries contained
	in $\bigoplus\limits_{Q \in W} \mathcal{E}_Q(0,r_Q]$.
	From \eqref{local change of Frobenius eq with a-root} we know that
	$U_q \circ C^{\nu^{a-1}+ \dots + 1}$ and $U_p \circ C$
	act on $V$.
	Since $U_q \circ C^{\nu^{a-1}+ \dots + 1}=(U_p \circ C)^a$, 
	Lemma \ref{lemma: a-th root lemma} tells us that
	\begin{align}\label{equation: twisting space}
	\begin{split}
		NP_q(U_q\circ \alpha N|\bigoplus\limits_{j=0}^{a-1} \mathcal{B}(0,\mathbf{r}])&=NP_q(U_q \circ C^{\nu^{a-1}+ \dots + 1}|V) \\
		&= \frac{1}{a} NP_p(U_p \circ C|V).
	\end{split}
	\end{align}

	\begin{proposition} \label{local to global: kernel and cokernel}
		The following hold:
		\begin{enumerate}
			\item We have $pr(\widehat{V}_0)= \mathcal{O}_{\mathcal{R}^{trun}}$, where $pr$ is the projection map
			defined in \S \ref{paragraph: global to semi-local setup}.
			\item Both $\ker(pr:V \to \mathcal{R}^{trun})$ and
			$\ker(pr:\widehat{V} \to \mathcal{R}^{trun})$ have dimension $a(g-1+r_0+r_1+r_\infty-\Omega_\rho)$ as vector
			spaces over $L$.
		\end{enumerate}
	\end{proposition}

	\noindent To prove Proposition \ref{local to global: kernel and cokernel}
	we need the following lemma:
	
	\begin{lemma} \label{computing kernel dimensions by reducing mod pi}
		Let $f: R \to S$ be a continuous map of Banach spaces
		such that $f(R_0) \subset S_0$. If $\overline{f}: \overline{R} \to \overline{S}$ is surjective, 
			then $f$ is surjective and $f(R_0)=S_0$. Furthermore,
			\begin{align*}
			\overline{\ker(f)} &= \ker(\overline{f}).
			\end{align*}
	\end{lemma}
	\begin{proof}
		This is proven by approximating the image and kernel of $f$.
		For more details see \cite[Lemma 7.3]{kramermiller-padic}. 
	\end{proof}

	\begin{proof} (of Proposition \ref{local to global: kernel and cokernel})
		Let us first consider $\widehat{V}$.  Define a function $\mu: W \to \mathbb{N}$
		by
		\begin{align}\label{eq: def of mu}
		\mu(Q) &= \begin{cases}
		1 & \eta(Q)\in \{0,\infty\} \\
		p & \eta(Q) = 1
		\end{cases}.
		\end{align}
		Let $\overline{M}$
		 be the reduction of $M$
		modulo $\gothm$. By Lemma \ref{computing kernel dimensions by reducing mod pi} 
		and \eqref{local change of frob is eq 1 mod p} we may prove 
		the corresponding result for the map
		\begin{align} \label{projection map mod p equation}
			\overline{pr}: \overline{M}\Big(\bigoplus_{j=0}^{a-1} \overline{B}\Big) \to
			\overline{\mathcal{R}^{trun}}=\bigoplus_{Q \in W} \bigoplus_{j=0}^{a-1}  u_{Q,j}^{-\mu(Q)}\mathbb{F}_q\llbracket u_{Q,j}^{-1}\rrbracket.
		\end{align} 
		Define the divisor
		\begin{align*}
		D_j &= \sum_{i=1}^{r_1} (p-1)[P_{1,i}] - \sum_{Q \in W}n_{Q,j}[Q].
		\end{align*}  
		By \eqref{local change of frob is eq 1 mod p} we know
		the kernel of \eqref{projection map mod p equation} is 
		\[ \bigoplus_{j=0}^{a-1} H^0(X, \mathcal{O}_X(D_j)).\]
		Since $(p-1)r_1=\deg(\eta)$ we know from \eqref{equation: sum of nonexponent divisor} that 
		$\deg(D_i)\geq \deg(\eta)-r_0-r_\infty$.
		By \eqref{riemann-hurwitz eq} and the Riemann-Roch theorem, we see
		$H^0(X, \mathcal{O}_{X}(D_j))$ has dimension $g-1+r_0+r_1+r_\infty - \sum n_{Q,j}$.
		Then from \eqref{equation: sum of nonexponent divisor 2} we know
		the kernel of \eqref{projection map mod p equation} has dimension
		$a(g-1+r_0+r_1+r_\infty-\Omega_\rho)$ as an $\mathbb{F}_q$-vector space. 
		To prove the result for $V$, first note
		that 
		\begin{align*}
			\ker(pr:\mathcal{R} \to \mathcal{R}^{trun}) \subset \bigoplus_{Q \in W} \bigoplus_{j=0}^{a-1} \mathcal{E}_Q(0,r_Q],
		\end{align*}
		as the kernel consists of functions with finite order poles. The proposition follows from \eqref{eq: growth is the same before and after the twist}.
	\end{proof}

	\paragraph{Choosing a basis.}
	For the remainder of this section, we let $v=U_p \circ C$, which we view as an operator on $V$. 
	Then define $J \subset \mathbb{N} \times W\times \{0,\dots, a-1\}$ by
	\begin{align} \label{J definition}
	J &= \Big \{(n,Q,j) ~~ | ~~ n\geq \mu(Q), ~j \in \{0,\dots,a-1\} \Big \},
	\end{align}
	where $\mu$ is the function defined in \eqref{eq: def of mu}.
	The set $\{u_{Q,j}^{-n}\}_{(n,Q,j) \in J}$ is an orthonormal basis 
	for $\mathcal{R}^{trun}$ over $L$ (recall that $u_{Q,j}$ is the element of $\mathcal{R}$
	with $u_Q$ in the $(Q,j)$-coordinate and zeros in the other coordinates). Let $K$ be a set with $\dim_L(\ker_L(pr|_V))$ elements and set
	$I=J \sqcup K$. 
	For $i=(n,Q,j) \in J$, choose an element $e_i \in V_0$ 
	with $pr(e_i)=u_{Q,j}^{-n}$. By the first part of Lemma \ref{local to global: kernel and cokernel}
	we know that such an $e_i$ exists.
	We also choose an orthonormal basis $\{e_i\}_{i \in K}\subset V_0$ of $\ker_L(pr|_V)$
	indexed by $K$. Then $G=\{e_i\}_{i \in I}$ is an orthonormal basis of $\widehat{V}$ over $L$. 
	By \eqref{equation: projecting onto tails ends up in W} there exists $c_i \in \mathcal{O}_{\mathcal{R}}^{con}$
	for each $i \in I$ with
	\begin{align} \label{eq how the basis looks}
	e_i =\begin{cases}
	u_{Q,j}^{-n} + c_{i} & i=(n,Q,j) \in J \\
	c_i & i \in K.
	\end{cases} 
	\end{align}
	
	Define the space
	\begin{align*}
		V^{con}&= (\mathcal{O}_{\mathcal{R}}^{con} \cap \widehat{V})\otimes_{\Z_p} \Q_p.
	\end{align*}
	Endow $V^{con}$ with a norm so that $V^{con}_0=\mathcal{O}_{\mathcal{R}}^{con} \cap \widehat{V}$ (recall that
	the subscript $0$ denotes the subset of elements of norm $\leq 1$).
	We now scale each element $e_i \in G$ by an element $x_i$ in $\mathcal{O}_L$ to obtain a formal basis $G^{con}$ of $V^{con}$. 
	We break up the definition of $x_i$ into four cases: the
	first case is when $i \in K$ and the other three cases correspond to
	the three types of points $Q\in W$ described in \S \ref{subsubsection: the local frob and Up}. We define
	\begin{align} \label{X definition}
	x_i &= \begin{cases}
	1 & i \in K \\
	\pi_{a \mathfrak{s}_Q}^{q(\mathbf{e}_Q,j)} \pi_{\mathfrak{s}_Q}^{pn} & i=(n,Q,j),~ \eta(Q) \in\{0,\infty\}  \text{ and $\rho^{wild}_Q$ is unramified} \\
	\pi_{a s_Q}^{q(\mathbf{e}_Q,j)} \pi_{s_Q}^{pn} & i=(n,Q,j),~ \eta(Q) \in\{0,\infty\}  \text{ and $\rho_Q^{wild}$ is ramified} \\
	p^{b(n)} & i=(n,Q,j) \text{ and } \eta(Q)=1,
	\end{cases}
	\end{align}
	where $b(n)$ is the function defined in \S \ref{subsection: type 2 frobenius}.
	From the definition of $\mathcal{O}_{\mathcal{R}}^{con}$ we see that
	$G^{con}=\{x_ie_i\}$ is a formal basis of $V^{con}$.
	Indeed, we just selected the $x_i$ appropriately for each summand in the definition
	of $\mathcal{O}_{\mathcal{R}}^{con}$.
	
	\begin{proposition}\label{proposition: compute with Bcon}
		We have
		\begin{align*}
		\det_E(1-sU_p\circ C|V) &= \det_E(1-sU_p\circ C| G^{con}_E).
		\end{align*}
	\end{proposition}
	\begin{proof}
		For $Q \in W$, define a sequence $b_{Q,1},b_{Q,2}, \dots \in \mathcal{O}_L$ such that
		$\{\dots, u_Q^2,u_Q^1,1,b_{Q,1}u_Q^{-1}, b_{Q,2}u_Q^{-2}, \dots\}$ is a formal basis of
		$\mathcal{E}_Q(0,r_Q]$. 
		For $i\in K$ set $y_i=1$ and for $i=(n,Q,j)$ set $y_i=b_{Q,n}$. Then $G^{\mathbf{r}}=\{y_ie_i\}$
		is an orthonormal basis of $V$. In particular, we have 
		\begin{align*}
			\det_E(1-sU_p\circ C| V)&=\det_E(1-sU_p\circ C| G_E^{\mathbf{r}}) \\
			&= \det_E(1-sU_p\circ C| G^{con}_E).
		\end{align*}
		The second equality follows by observing that the matrices of $U_p\circ C$ 
		for the bases $G_E^{\mathbf{r}}$ and 
		$G^{con}_E$ are similar.
	\end{proof}

	\paragraph{Estimating the column vectors.}
	To estimate the column vectors we will need the following lemma.
	\begin{lemma}\label{lemma: divisibility in Rconv}
		For any $n\geq 0$ we have $\pi_\circ^n \mathcal{O}_{\mathcal{R}}^{con} \cap \widehat{V} = \pi_\circ^nV^{con}_0$.
	\end{lemma}
	\begin{proof}
		Let $z \in \pi_\circ^n \mathcal{O}_{\mathcal{R}}^{con} \cap \widehat{V}$. Then $\pi_\circ^{-n}z \in \mathcal{O}_{\mathcal{R}}^{con}$ and since
		$\widehat{V}$ is a vector space we have $\pi_\circ^{-n}z \in \widehat{V}$. It follows that $z \in \pi_\circ^n(\mathcal{O}_{\mathcal{R}}^{con} \cap \widehat{V})$.
		The other direction is similar.
	\end{proof}
	We now estimate $\mathbf{col}_{(i,1)}(v,G_E^{con})$ for each $i \in I$. 
	We break this up into the four cases used when defining $x_i$.
	\begin{enumerate}[label=(\Roman*)]
		\item For $i \in K$, we have $x_ie_i=e_i$. We know from \eqref{eq how the basis looks} that
		$e_i \in \mathcal{O}_{\mathcal{R}}^{con}$. By \eqref{equation: W is preserved} we know $v(\mathcal{O}_{\mathcal{R}}^{con}) \subset \mathcal{O}_{\mathcal{R}}^{con}$, which means
		$v (e_i) \in V_0^{con}$. Thus, $\mathbf{col}_{(i,1)}(v,G_E^{con})\geq 0$
		and 
		\begin{align*}
			\{\mathbf{col}_{(i,1)}(v,G_E^{con})\}_{i \in K} \succeq \{\underbrace{0,0,\dots,0}_{a(g-1+r_0+r_1+r_\infty-\Omega_\rho)} \}.
		\end{align*}
		The multiplicity of the zeros follows from Lemma \ref{local to global: kernel and cokernel}.
		\label{Column: case 1}
		\item Fix $Q$ with $\eta(Q)=1$ and let $i=(n,Q,j)\in J$. By \eqref{J definition}, we only consider tuples $(n,Q,j)$ with $n\geq p$. 
		Recall from \eqref{X definition} that $x_i=p^{b(n)}$ and from \eqref{eq how the basis looks} that $e_i=u_{Q,j}^{-n} + c_i$ with $c_i \in \mathcal{O}_{\mathcal{R}}^{con}$. 
		Write $n=k+pm$, where $0 \leq k <p$. By \eqref{equation: W is preserved} we have $v(x_ic_i) \in p^{b(n)}\mathcal{O}_{\mathcal{R}}^{con}$ 
		and by \eqref{UP computation type 1} we have $v(x_iu_{Q,j}^{-n}) \subset p^m\mathcal{O}_{\mathcal{R}}^{con}$.
		From the definition of $b(n)$ in \S \ref{subsection: type 2 frobenius}, we know $b(n)\geq m$, which implies
		$v (x_ie_i) \in p^m\mathcal{O}_{\mathcal{R}}^{con}$. Lemma \ref{lemma: divisibility in Rconv} tells us that
		$v (x_ie_i) \in p^m(V_0^{con})$. Thus, we have $\mathbf{col}_{(i,1)}(v,G_E^{con})\geq m$.
		This gives:
		\begin{align*}
			P_Q=\{\mathbf{col}_{((n,Q,j),1)}(v,G_E^{con})\}_{\substack{n\geq p \\ 0\leq j < a}} \succeq  \{1,2,3,\dots \}^{ \times ap}.
		\end{align*}
		
		\label{Column: case 2}
		\item Fix $Q\in W$ such that $\eta(Q)\in\{0,\infty\}$ and $\rho^{wild}_Q$
		is unramified. Consider $i=(n,Q,j)\in J$. By \eqref{J definition} we only consider tuples $(n,Q,j)$ where
		$n\geq 1$. Recall from \eqref{X definition} that $x_i=\pi_{a \mathfrak{s}_Q}^{q(\mathbf{e}_Q,j)} \pi_{\mathfrak{s}_Q}^{pn}$ and from \eqref{eq how the basis looks} that $e_i=u_{Q,j}^{-n} + c_i$ with $c_i \in \mathcal{O}_{\mathcal{R}}^{con}$. 
		Then by \eqref{UP computation type 2} and \eqref{equation: W is preserved}, we see that $v(x_ie_i) \in \pi_{\mathfrak{s}_Q}^{n(p-1)} \pi_{a\mathfrak{s}_Q}^{-\omega_Q} \mathcal{O}_{\mathcal{R}}^{con}$. Again, by Lemma \ref{lemma: divisibility in Rconv} we see that
		$v (x_ie_i) \in \pi_{\mathfrak{s}_Q}^{n(p-1)} \pi_{a\mathfrak{s}_Q}^{-\omega_Q} (V_0^{con})$. This gives:
		\begin{align*}
		P_Q=\{\mathbf{col}_{((n,Q,j),1)}(v,G_E^{con})\}_{\substack{n\geq 1 \\ 0\leq j < a}} \succeq  \Bigg\{\frac{1}{\mathfrak{s}_Q}-\frac{\omega_Q}{a\mathfrak{s}_Q(p-1)},\frac{2}{\mathfrak{s}_Q}-\frac{\omega_Q}{a\mathfrak{s}_Q(p-1)}, \dots \Bigg\}^{\times a}.
		\end{align*}
		
		\label{Column: case 3}
		\item Finally,  fix $Q\in W$ such that $\eta(Q)\in\{0,\infty\}$ and $\rho^{wild}_Q$
		is ramified. Repeating the argument from \ref{Column: case 3} where we replace $\mathfrak{s}_Q$ with $s_Q$ gives:
		\begin{align*}
		P_Q=\{\mathbf{col}_{((n,Q,j),1)}(v,G_E^{con})\}_{\substack{n\geq 1 \\ 0\leq j < a}} \succeq  \Bigg\{\frac{1}{s_Q}-\frac{\omega_Q}{as_Q(p-1)},\frac{2}{s_Q}-\frac{\omega_Q}{as_Q(p-1)}, \dots \Bigg\}^{\times a}.
		\end{align*}
		
		\label{Column: case 4}
	\end{enumerate}
	\noindent 
	We put everything together to get
	\begin{align*}
		\{\mathbf{col}_{(i,1)}(v,G_E^{con})\}_{i \in I} \succeq \{\underbrace{0,0,\dots,0}_{g-1+r_0+r_1+r_\infty-\Omega_\rho} \}^{\times a}
		\bigsqcup  \Bigg (\bigsqcup_{Q \in W} P_Q \Bigg ).
	\end{align*}
	Then by Lemma \ref{lemma: estimating NP by estimating columns} we see that
	$\det(1-sv,G_E^{con}$ converges and that
	\begin{align*}
		NP_p(v|G_E^{con}) &\succeq \{\underbrace{0,0,\dots,0}_{g-1+r_0+r_1+r_\infty-\Omega_\rho} \}^{\times a^2}
		\bigsqcup  \Bigg (\bigsqcup_{Q \in W} P_Q^{\times a} \Bigg ).
	\end{align*}
	Then from Proposition \ref{proposition: compute with Bcon} we have
	\begin{align*}
		\frac{1}{a} NP_p(U_p \circ C| V) &\succeq \{\underbrace{0,0,\dots,0}_{g-1+r_0+r_1+r_\infty-\Omega_\rho} \}^{\times a}
		\bigsqcup  \Bigg (\bigsqcup_{Q \in W} P_Q^{} \Bigg ).
	\end{align*}
	When $Q$ is from case \ref{Column: case 2} each slope in $P_Q$ is at least one.
	Also, when $Q$ is from case \ref{Column: case 3}, we know
	from \eqref{equation: false wild ramification number} that each slope in $P_Q$ is at least one. This gives 
	\begin{align*}
	\frac{1}{a} NP_p(U_p \circ C|V)_{< 1} &\succeq \big 
	\{\underbrace{0,\dots,0}_{g-1+r_0+r_1+r_\infty-\Omega_\rho}
	\big \}^{\times a}
	\bigsqcup \Bigg ( \bigsqcup_{i=1}^{\mathbf{m}} S_{\tau_i}^{\times a} \Bigg ).
	\end{align*}
	Proposition \ref{proposition: NP bounds with operator assumption1}
	follows from \eqref{equation: twisting space}.

		\subsection{Finishing the proof}
		
		We now finish the proof of Theorem \ref{main theorem}. 
		From \eqref{equation: L function estimate comes from U_p}
		and Proposition \ref{proposition: NP bounds with operator assumption1}
		we know
		\begin{align*} 
		NP_q(L(\rho,V,s))_{<1} & \succeq \big 
		\{\underbrace{0,\dots,0}_{g-1+r_0+r_1+r_\infty-\Omega_\rho}
		\big \}
		\bigsqcup \Bigg ( \bigsqcup_{i=1}^{\mathbf{m}} S_{\tau_i} \Bigg ).
		\end{align*}
		Comparing \eqref{introduction of L-function} with \eqref{introduction of L-function:2}
		gives
		\begin{align*} 
		L(\rho,V,s)&= L(\rho,s)\cdot \prod_{\substack{Q \in W  \\ Q \neq \tau_i}} 
		(1-\rho(Frob_Q)s).
		\end{align*}
		This product has $r_0+r_1+r_\infty-{\mathbf{m}}$ terms, each accounting for a 
		slope zero segment. Thus,
		\begin{align*} 
		NP_q(L(\rho,s))_{<1} \succeq  \big 
		\{\underbrace{0,\dots,0}_{g-1+{\mathbf{m}}-\Omega_\rho}
		\big \}
		\bigsqcup \Bigg ( \bigsqcup_{i=1}^{\mathbf{m}} S_{\tau_i} \Bigg ).
		\end{align*}	
		From the Euler-Poincare formula (see e.g. \cite{Raynaud-euler_poincare}) we know $L(\rho,s)$ has degree 
		$2(g-1+{\mathbf{m}}) + \sum (s_{\tau_i} - 1)$. This accounts for the remaining slope one 
		segments. The proof is complete.

		\printnoidxglossary[type=symbols,style=long,title={List of Symbols}]
		
	\bibliographystyle{plain}
	\bibliography{bibliography.bib}

\end{document}